\documentclass[twoside,leqno]{article}

\usepackage[ruled,vlined]{algorithm2e}

\usepackage[letterpaper]{geometry}

\usepackage{siamproceedings}

\usepackage[T1]{fontenc}
\usepackage{bbm}
\usepackage{amsfonts}
\usepackage{graphicx}
\usepackage{epstopdf}
\usepackage{enumitem}
\usepackage{algorithmic}
\ifpdf
  \DeclareGraphicsExtensions{.eps,.pdf,.png,.jpg}
\else
  \DeclareGraphicsExtensions{.eps}
\fi

\newsiamremark{remark}{Remark}
\newsiamremark{hypothesis}{Hypothesis}
\crefname{hypothesis}{Hypothesis}{Hypotheses}
\newsiamthm{claim}{Claim}

\newtheorem{assumption}[theorem]{Assumption}

\newcommand{\R}{\mathbb{R}}
\usepackage{xcolor}

\usepackage{amsopn}

\footnotetext[3]{School of Industrial and Systems Engineering, Georgia Tech}

\begin{document}

\newcommand\relatedversion{}
\renewcommand\relatedversion{\thanks{The full version of the paper can be accessed at \protect\url{https://arxiv.org/abs/0000.00000}}} 

\title{On Complexity of Model-Based Derivative-Free Methods. }
\author{A. Chaudhry\footnotemark[3]   \and K. Scheinberg\footnotemark[3]}

\date{}

\maketitle







\begin{abstract} In many applications of mathematical optimization, one may wish to optimize an objective function without access to its derivatives.  These situations call for derivative-free optimization (DFO) methods. 
Among the most successful approaches in practice are model-based trust-region methods, such as those pioneered by M.J.D Powell. While relatively complex to implement, these methods are now available in standard scientific computing platforms, including MATLAB and SciPy. However, theoretical analysis of their computational complexity lags behind practice. In particular, it is important to bound the number of function evaluations required to achieve a desired level of accuracy. In this paper we systematically  derive complexity bounds for classical model-based trust-region methods and their modern variations. We establish, for the first time, that these methods can have the same worst case complexity than any other known DFO method. 
\end{abstract}
\section{Introduction.}
Derivative Free Optimization (DFO) is an area of optimization that concerns with developing optimization methods based purely on function value computation without applying any direct differentiation. Other related areas of optimization are called {\em zeroth-order} optimization (because it relies of zeroth-order oracles) and  black-box optimization. The applications of DFO are numerous and rapidly growing with the sophistication of engineering models. More and more systems are being modeled and evaluated by complex computer codes, which brings the natural next step - optimizing such systems. The application areas range from well established,  such as aerospace engineering, chemical engineering, civil and environmental engineering to more recent, such as reinforcement learning, machine learning and simulation optimization. 
A broad list of applications can be found in \cite{audet2017derivativefree} and \cite{rios2013derivativefree}. There is a rich literature of DFO starting from around mid 90s which is rapidly growing with new interest spurred by new applications in engineering, machine learning and artificial intelligence.  Aside from an increasing number of papers, there are two books \cite{DFOBook} and \cite{audet2017derivativefree}, and a survey on the topic \cite{LarsMeniWild2019}. 
 
There are three major classes of DFO algorithms: 1. {\bf directional direct search} methods, 2.  gradient descent  based on {\bf simplex gradients} which includes finite difference approximation and 3. {\bf interpolation based (also known as model based) trust region methods}. The first two classes are rather simple to implement and analyze which makes them popular choices, especially in the literature and with recent applications to machine learning where they are easy to adopt. They operate by sampling a certain number of  function values at some chosen points around the current iterate and then make a step according to the information obtained from the samples. The way the sample sets are constructed is predetermined (but can be random) and is meant to ensure that a descent direction is identified either deterministically or in expectation. Such  method is then analyzed based on this "guaranteed descent"  and the effort it takes to obtain it. 

 The third type of methods pioneered in the 90s by M.J.D Powell \cite{UOBYQA,MJDPowell_2004b}, is much more complex, yet is arguably the most effective in practice for many applications \cite{more2009benchmarking}. 
 These methods trade off exploration and exploitation by reusing function values computed in the past iterates and adapting where the function should be sampled next based on information accumulated so far.  The complex structure of the algorithms, especially as presented in Powell's papers, resulted in scarcity of  theoretical analysis, especially in terms of complexity. Some underlying theory of asymptotic convergence was developed in \cite{DFOBook}, however, methods analyzed there  and in subsequent works rely on a "criticality step" for the theory to work, which departs from the main idea of the algorithms. There is still a lack of understanding of how these methods work and most importantly if they enjoy favorable complexity bounds. 

In this paper, we focus on the complexity of model-based derivative free algorithms that aim to solve unconstrained nonlinear optimization problems of the form
\begin{align} \label{eq.prob}
	\min_{x \in \mathbb{R}^n} \phi(x), 
\end{align}
where $\phi: \mathbb{R}^n \rightarrow \mathbb{R}$ is a smooth objective function that may or may not be convex. The key premise of these methods is to economize on function values, possibly at the expense of additional linear algebra, since in most applications the function evaluations cost dominates all else. The complexity will be measured in terms of the number of function evaluations needed to achieve an $\epsilon$-stationary point, that is a point $x$ for which $\|\nabla \phi(x)\|\leq \epsilon$. 

Throughout the paper, we assume that we have access to a zeroth-order oracle $f(x)\approx {\phi}(x)$ which may or may not be exact.
In practice, one may wish to tolerate an inexact oracle but we will sometimes treat the exact case for ease of exposition.
The following are the standard assumptions on $\phi(x)$ for our setting. 
\begin{assumption}[\textbf{Lower bound on $\pmb{\phi}$}] \label{assum:low_bound} 
	The function $\phi$ is bounded below by a scalar $\phi^\star$ on $\mathbb{R}^n$. 	
\end{assumption}

\begin{assumption}[\textbf{Lipschitz continuous gradient}]	\label{assum:lip_cont}  
	The function $\phi$ is continuously differentiable, and the gradient of $\phi$ is $L$-Lipschitz continuous on $\mathbb{R}^n$, i.e., $\|\nabla \phi(y) - \nabla \phi(x)\| \le L \|y-x\|$ for all $(y,x) \in \mathbb{R}^n\times\mathbb{R}^n$.
\end{assumption}

The paper is organized as follows. In Section \ref{sec:ffd} we present the basic model based trust region method and its complexity analysis for the case when models are based on simple finite difference gradient approximations. In Section \ref{sec:powell} we introduce a "geometry-correcting" trust region method in the spirit of those proposed by Powell. We will 
show how the analysis of complexity can be adapted from that of the basic method to the more sophisticated one. In Section \ref{sec:lagrange} and \ref{sec:powell2} we derive new results that help us establish that the complexity of the "geometry-correcting" trust region method is competitive with that of the basic trust region method. Finally, in Section \ref{sec:subspace_ffd} we develop and analyze model based trust region methods in random subspaces and show that these methods have complexity bounds that are as good as those for any other known DFO method.

\section{Basic trust-region method and analysis.}
\label{sec:ffd}
We first present and analyze a basic trust region method which, at every iteration $k \in \{0,1,\dots\}$, 
constructs a quadratic model to approximate $\phi(x)$ near the iterate $x_k$
\begin{equation}\label{eq:model_def}
	m_k(x_k+s) = {\phi}(x_k) +  g_k(x_k)^Ts  + \frac{1}{2} s^T H_k(x_k) s. 
\end{equation}
The model is then minimized (approximately) over the trust region $B(x_k,\Delta_k)$ - a Euclidean ball around $x_k$ of a radius $\Delta_k$. In what follows we use the abbreviation $g_k:=g(x_k)= \nabla m_k(x_k)$ and $H_k:= H(x_k)=\nabla^2 m_k(x_k)$.\footnote{Note that ${\phi}(x_k)$ is used in the definition \eqref{eq:model_def} but not in minimization of $m_k$, thus there is no need to compute it. }

\begin{algorithm}[ht] 
    \caption{~\textbf{Trust region method based on fully-linear models}}
    \label{alg:tr}
       {\bf Inputs:} A zeroth-order oracle $f(x)\approx {\phi}(x)$, a starting point $x_0$, TR radius $\Delta_0$, and hyperparameters $\eta_1\in(0,1)$, $\eta_2 > 0$, and $\gamma\in(0,1)$.  \\
      \For{$k=0,1,2,\cdots$}{    
        \nl Compute model $m_k$ and a trial step $x_k+s_k$ where $s_k\approx \arg\min_s \{m_k(x_k+s):~s \in B(0,\Delta_k)\}$.\\
        \nl Compute the ratio $\rho_k$ as 
        \[ \rho_k= \frac{{f}(x_k) - {f}(x_k+s_k)}{m(x_k) - m(x_k+s_k)}.
        \] 
        \nl Update the iterate and the TR radius as 
        \[  (x_{k+1}, \Delta_{k+1}) \gets \left\{ \begin{aligned} 
            &(x_k+s_k, \gamma^{-1} \Delta_k) &&\text{if } \rho \ge \eta_1 \text{ and } \|g_k\| \ge \eta_2 \Delta_k, \\
            &(x_k, \gamma \Delta_k) &&\text{otherwise. }
        \end{aligned} \right. 
        \]
      }
\end{algorithm}


Algorithm \ref{alg:tr} is essentially a standard trust region method, well studied in the literature \cite{TRbook}, except for the condition $\|g_k\| \ge \eta_2 \Delta_k$. This condition is not used in classical TR methods where $\nabla m_k(x_k)=g_k=\nabla \phi(x_k)$. However,  for TR methods based on models constructed in derivative free setting and, more generally,  models based on inexact gradient estimates, the TR radius $\Delta_k$ serves two roles - that of the step size parameter and  that of the control of the gradient accuracy. Condition $\|g_k\| \ge \eta_2 \Delta_k$ is essential for the latter role to ensure that the gradient accuracy stays on track as the norm of the gradient reduces. We note here that much of prior DFO literature replaces this condition with a much less practical and cumbersome "criticality step". Thus one of the contributions of  this paper is performing analysis of  Algorithm \ref{alg:tr} and and other  model-based trust region algorithms without the "criticality step".

   For all TR algorithms studied in this paper 
 we will make the following standard assumption on the models $m_k$ and their minimization.

\begin{assumption} \label{assum:tr}
    \begin{enumerate}
\item The trust region subproblem is solved sufficiently accurately in each iteration $k$ so that $x_k+s_k$ provides at least a fraction of Cauchy decrease, i.e. for some constant $0 < \kappa_{fcd} < 1$ 
          \begin{equation}\label{eq:Cauchy decrease}
            m_k(x_k) - m_k(x_k+s_k) \ge \frac{\kappa_{fcd}}{2} \|g_k\| \min\bigg\{\frac{\|g_k\|}{\|H_k\|}, \Delta_k\bigg\}. 
          \end{equation}
        \item There exists a constant $\kappa_{bhm} > 0$ such that, for all $x_k$ generated by Algorithm~\ref{alg:tr}, the spectral norm of the Hessian of the model 
          \[ \|H_k\| \le \kappa_{bhm}. \] 
    \end{enumerate}
\end{assumption}

Condition \eqref{eq:Cauchy decrease} is commonly used in the literature and is satisfied by the Cauchy point with $\kappa_{\rm fcd} = 1$. 
See \cite[Section 6.3.2]{TRbook} for more details.

The following definition, also widely used in the literature  \cite{DFOBook}, helps us  identify the requirements on models $m_k$  that are critical for convergence.

\begin{definition}[Fully-linear model] \label{def:fully-linear}
Given a ball  around point $x$ of radius $\Delta$, $B(x,\Delta)$,  we say that model $m(x+s)$ is a $\kappa_{ef}, \kappa_{eg}$-fully linear model of $\phi(x+s)$  on 
$B(x, \Delta)$ if
\[
\|\nabla m(x)-\nabla \phi(x)\|\leq \kappa_{eg}\Delta
\]
and 
\[
|m(x+s)-\phi(x+s)|\leq \kappa_{ef}\Delta^2
\]
for all $\|s\|\leq \Delta$. 
\end{definition}

We temporarily make the following additional assumptions for the analysis of  Algorithm \ref{alg:tr}.

\begin{assumption} \label{assum:exact}
 At each iteration ${f}(x_k)-{f}(x_k+s_k)=\phi(x_k)-\phi(x_k+s_k)$, in other words, exact function reduction is computed, when computing $\rho_k$. 
 \end{assumption}

The next lemma is key in the analysis of any trust region algorithm and specifically Algorithm~\ref{alg:tr}. It establishes that once $\Delta_k$ is sufficiently small compared to the gradient norm, a successful step is ensured. 
\begin{lemma}[small $\Delta_k$ implies successful step] \label{lem:tr_success}
    Under Assumptions~\ref{assum:tr} and {assum:exact}, if $m_k$ is $\kappa_{ef}, \kappa_{eg}$-fully-linear and 
    \begin{equation} \label{eq:tr_success}
        \Delta_k \le C_1 \|\nabla \phi(x_k)\| \text{, where } C_1 = \bigg(\max\left\{\eta_2,\ \kappa_{bhm},\ \frac{2\kappa_{ef}}{(1-\eta_1)\kappa_{fcd}}\right\}+\kappa_{eg}\bigg)^{-1}, 
    \end{equation}
    then $\rho \ge \eta_1$, $\|g_k\| \ge \eta_2 \Delta_k$, thus iteration $k$ is successful and  $x_{k+1} = x_k+s_k$. 
\end{lemma}
The  proof is a simplification of those in the stochastic trust region literature where condition $\|g_k\| \ge \eta_2 \Delta_k$ is widely used \cite{blanchet2019convergence, cao2023first}. 

\begin{proof}
    By the assumption that $m_k$ is fully-linear, we have  
    \[ \|\nabla {\phi}(x_k)\| \le \|g_k\| + \kappa_{eg}\Delta_k. 
    \]
   From \eqref{eq:tr_success} we conclude that 
    \[ \begin{aligned}
        (\max\{\kappa_{bhm},\eta_2\} + \kappa_{eg}) \Delta_k &\le \|\nabla {\phi}(x_k)\| \le \|g_k\| + \kappa_{eg}\Delta_k \\
        \max\{\kappa_{bhm},\eta_2\} \Delta_k &\le \|g_k\|. 
    \end{aligned} \] 
    This implies that the first condition of the successful step, namely $\|g_k\| \ge \eta_2\Delta_k$, is satisfied. 
    From \eqref{eq:Cauchy decrease} we have 
    $m_k(x_k) - m_k(x_k+s_k) \ge \kappa_{fcd} \|g_k\| \Delta_k / 2$. 
    Thus, using  the assumption that $m_k$ is fully-linear again,
    \[ \begin{aligned}
        \rho_k &= \frac{m(x_k) - m(x_k+s_k) + ({\phi}(x_k) - m(x_k)) - ({\phi}(x_k+s_k) - m(x_k+s_k))}{m(x_k) - m(x_k+s_k)} \\
        &\ge 1 - \frac{\kappa_{ef} \Delta_k^2}{m(x_k) - m(x_k+s_k)} 
        \ge 1 - \frac{\kappa_{ef} \Delta_k^2}{\kappa_{fcd} \|g_k\| \Delta_k / 2} \ge 1 - \frac{2\kappa_{ef} \Delta_k}{\kappa_{fcd} (\|\nabla {\phi}(x_k)\| - \kappa_{eg}\Delta_k)} \ge \eta_1, 
    \end{aligned} \]
    where the last step is true because $\|\nabla {\phi}(x_k)\| \ge \big(\frac{2\kappa_{ef}}{(1-\eta_1)\kappa_{fcd}} + \kappa_{eg}\big) \Delta_k$ follows from \eqref{eq:tr_success}. 
\end{proof}

\begin{lemma}[successful iteration implies function reduction] \label{lem:tr_progress}
    Let Assumptions~\ref{assum:lip_cont} and \ref{assum:tr} hold. 
    If  the iteration $k$ is successful, then   
    \begin{equation}\label{eq:tr_progress} 
        \phi(x_k) - \phi(x_{k+1}) \ge C_2\Delta_k^2 \text{, where } C_2 = \frac{\eta_1\eta_2\kappa_{fcd}}{2} \min\left \{\frac{\eta_2}{\kappa_{bhm}}, 1\right \}; 
    \end{equation}
    otherwise, we have $x_{k+1} = x_k$ and ${\phi}(x_k) - {\phi}(x_{k+1}) = 0$.  
\end{lemma}
\begin{proof}
    When the trial point is accepted, we have $\rho_k \ge \eta_1$ and $\|g_k\| \ge \eta_2 \Delta_k$, which guarantees 
    \begin{align*}
        {\phi}(x_k) - {\phi}(x_{k+1}) 
        &=   {\phi}(x_k) - {\phi}(x_k+s_k) ={f}(x_k) - {f}(x_k+s_k) 
        \ge \eta_1 \big(m(x_k) - m(x_k^+)\big) \\
        &\ge \frac{\eta_1 \kappa_{fcd}}{2} \|g_k\| \min\bigg\{\frac{\|g_k\|}{\|H_k\|}, \Delta_k\bigg\} 
        \ge \frac{\eta_1 \kappa_{fcd}}{2} \eta_2\Delta_k \min\bigg\{\frac{\eta_2\Delta_k}{\kappa_{bhm}}, \Delta_k\bigg\}. 
    \end{align*}
\end{proof}

    For a given $\epsilon>0$, let $K_\epsilon$ be the first iteration of Algorithms \ref{alg:tr} for which $\|\nabla f(x_k)\| \le \epsilon$.  Define the index sets
    \begin{equation}\label{eq:sdef}
     {\cal S}_{\epsilon} := \{k\in\{0,\dots,K_\epsilon-1\}:~ \text{iteration $k$ is successful}\}.
    \end{equation}
      \begin{equation}\label{eq:udef}
       {\cal U}_{\epsilon} := \{k\in\{0,\dots,K_\epsilon-1\}:~ \text{iteration $k$ is unsuccessful}\}.
     \end{equation}

  \begin{lemma}[Lower bound on $\Delta_k$]\label{lem:delta_bnd}
Assuming $m_k$ is $\kappa_{ef}, \kappa_{eg}$-fully linear for each $k=0, \ldots, K_\epsilon-1$, $\Delta_k \ge \gamma C_1 \epsilon$ for all $k \in \{0,\dots, K_\epsilon - 1\}$.
\end{lemma}
\begin{proof}
    According to Lemma~\ref{lem:tr_success}, any iteration $k \in \{0,\dots, K_\epsilon - 1\}$ must be successful when $\Delta_k \le C_1\epsilon$. 
    Thus, given $\Delta_0 \ge \gamma C_1 \epsilon$, and by the mechanism of Algorithm \ref{alg:tr} we must have $\Delta_k \ge \gamma C_1 \epsilon$ for all $k \in \{0,\dots, K_\epsilon - 1\}$.
  \end{proof}  

The following bound holds under the result of Lemma \ref{lem:delta_bnd}. 
\begin{lemma}[Bound of successful iterations]\label{lem:succ_iter_bnd}
From $\Delta_k \ge \gamma C_1 \epsilon$ for all $k \in \{0,\dots, K_\epsilon - 1\}$ we have
 \[
    |{\cal S}_\epsilon| \leq \frac{f(x_0)-f^\star}{C_2 (\gamma C_1 \epsilon)^2 }
\]
\end{lemma}
\begin{proof}
    Using Lemma~\ref{lem:tr_progress} we have  
    \begin{align*} 
        \phi(x_0) - \phi^\star &\ge \sum_{k=0}^{K_\epsilon-1} {\phi}(x_k) - {\phi}(x_{k+1}) \ge \sum_{k\in{\cal S}_\epsilon} C_2 \Delta_k^2 >  |{\cal S}_\epsilon | C_2 (\gamma C_1 \epsilon)^2 
    \end{align*}
     which gives the result of the lemma.
\end{proof}

\begin{lemma}\label{lem:unsucc_iter_bnd}
   Assuming that the initial trust-region radius $\Delta_0 \ge \gamma C_1 \epsilon$, 
  \[
        |{\cal U}_\epsilon| \le  |{\cal S}_\epsilon| + \lceil{\Big(  \log_\gamma \frac{C_1\epsilon}{\Delta_0} \Big)}\rceil. 
 \]
\end{lemma}
\begin{proof}
We observe that 
    \[ 
    \Delta_{K_\epsilon} =  \gamma^{-|{\cal S}_\epsilon|} \gamma^{|{\cal U}_\epsilon|} \Delta_0  \geq   C_1 \epsilon, 
    \]  
    The last inequality follows from the fact that $ \Delta_{K_\epsilon-1}\geq  \gamma C_1 \epsilon$ and the $K_\epsilon-1$-th iteration must be successful. 
        Thus the number of unsuccessful iterations can be bounded using the number of successful ones rearranging the terms and  taking the logarithm. 
 \end{proof}

Summarizing the above lemmas the following complexity result immediately follows. 
\begin{theorem}\label{thm:tr_complexity}
    Let Assumptions~\ref{assum:lip_cont}, \ref{assum:tr} and \ref{assum:exact} hold. 
   Assuming that the initial trust-region radius $\Delta_0 \ge \gamma C_1 \epsilon$, and $m_k$ is $\kappa_{ef}, \kappa_{eg}$-fully linear for each $k=0, \ldots, K_\epsilon-1$, we have the bound 
    \begin{equation} \begin{aligned} 
        K_\epsilon &=|{\cal S}_\epsilon|+ |{\cal U}_\epsilon|\le \frac{2(f(x_0) - f^\star)}{C_2 (\gamma C_1 \epsilon)^2} + \log_\gamma \frac{C_1\epsilon}{\Delta_0} \\ 
        &= \frac{4\kappa_{bhm} \Big(\max\left\{\eta_2,\ \kappa_{bhm},\ \frac{2\kappa_{ef}}{(1-\eta_1)\kappa_{fcd}}\right\}+\kappa_{eg}\Big)^2}{\gamma^2 \eta_1\eta_2\kappa_{fcd} \min\{\eta_2, {\kappa_{bhm}}\}} \cdot \frac{{\phi}(x_0) - {\phi}^\star}{\epsilon^2} + \log_\gamma \frac{C_1\epsilon}{\Delta_0}. 
    \end{aligned} \end{equation}
\end{theorem}

\subsection{The case of inexact  zeroth-order oracle.}
Now let us assume that for any $x$ the algorithm has access to a zeroth-order oracle that computes $f(x)$ such that 
   \[
   | f(x)-{\phi}(x)| \leq \epsilon_f. 
   \]
It is easy to extend Lemmas \ref{lem:tr_success} and \ref{lem:tr_progress} to this case.  The first modification will result in a slightly different definition in constant $C_1$, which we denote as $\bar C_1$ and which is 
  \begin{equation} \label{eq:tr_success_ef}
    \bar C_1 = \bigg(\max\left\{\eta_2,\ \kappa_{bhm},\ \frac{2\kappa_{ef}+4\epsilon_f}{(1-\eta_1)\kappa_{fcd}}\right\}+\kappa_{eg}\bigg)^{-1}. 
    \end{equation}
    The second modification will result in \eqref{eq:tr_progress} in the  statement of Lemma  \ref{lem:tr_progress} changing as follows

   \begin{equation}\label{eq:tr_progress_ef}
   \phi(x_k) - \phi(x_{k+1}) \ge C_2\Delta_k^2 -2\epsilon_f \text{, where } C_2 = \frac{\eta_1\eta_2\kappa_{fcd}}{2} \min\{\frac{\eta_2}{\kappa_{bhm}}, 1\}. 
    \end{equation}

    Thus, under the additional assumption that $\Delta_k\geq \sqrt{\frac{2\epsilon_f}{\tau C_2}} $ for some $\tau\in(0,1)$ \eqref{eq:tr_progress_ef} can be further stated as 
     \begin{equation}\label{eq:tr_progress_tau}
   \phi(x_k) - \phi(x_{k+1}) \ge (1-\tau) C_2\Delta_k^2 \text{, where } C_2 =  \frac{\eta_1\eta_2\kappa_{fcd}}{2} \min\{\frac{\eta_2}{\kappa_{bhm}}, 1\}.
    \end{equation}
   The new versions of Lemmas \ref{lem:tr_success} and \ref{lem:tr_progress} can be applied as before, as long as it can be ensured that    $\Delta_k$ remains 
   not smaller than $\sqrt{\frac{2\epsilon_f}{\tau C_2}}$. Due to Lemma \ref{lem:tr_success} and the update mechanism for $\Delta_k$ this is ensured as long as $
   \|\nabla \phi(x_k)\|\geq \epsilon$ for  sufficiently large $\epsilon$. 
\begin{theorem}\label{thm:tr_complexity_ef}
    Let Assumption~\ref{assum:lip_cont}  hold. 
 For any $\epsilon>  \sqrt{\frac{2\epsilon_f}{\gamma^2 \tau C_2\bar C_1^2}}$, assuming that the initial trust-region radius $\Delta_0 \ge \gamma \bar C_1 \epsilon$, where $\bar C_1$ is defined in \eqref{eq:tr_success_ef} and $m_k$ is $\kappa_{ef},\kappa_{eg}$-fully linear for each $k=0, \ldots, K_\epsilon-1$, 
where $ K_\epsilon$ is the first iteration  for which $\|\nabla \phi(x_k)\|\leq \epsilon$, then we have the bound
    \begin{equation} \begin{aligned} 
        K_\epsilon &=|{\cal S}_\epsilon|+ |{\cal U}_\epsilon|\le \frac{2(\phi(x_0) - \phi^\star)}{(1-\tau)C_2 (\gamma \bar C_1 \epsilon)^2} + \log_\gamma \frac{\bar C_1\epsilon}{\Delta_0} \\ 
        &= \frac{4\kappa_{bhm} \Big(\max\left\{\eta_2,\ \kappa_{bhm},\ \frac{2\kappa_{ef}+4\epsilon_f}{(1-\eta_1)\kappa_{fcd}}\right\}+\kappa_{eg}\Big)^2}{(1-\tau)\gamma^2 \eta_1\eta_2\kappa_{fcd} \min\{\eta_2, {\kappa_{bhm}}\}} \cdot \frac{{\phi}(x_0) - {\phi}^\star}{\epsilon^2} + \log_\gamma \frac{\bar C_1\epsilon}{\Delta_0}. 
    \end{aligned} \end{equation}
\end{theorem}

\subsection{Ensuring fully-linear models on each iteration and complexity implications.}
Let us now describe the most straightforward way of constructing fully linear models in derivative free setting. 
First we show that it can be achieved by  using a sufficiently accurate gradient approximation.  
The following lemma easily follows from Assumption \ref{assum:tr} and the smoothness of $\phi(x)$. 
\begin{lemma}[Fully linear models]\label{lem:fully-lin}
Under Assumption \ref{assum:tr}  if $\|\nabla m({x}) - \nabla {\phi}({x})\|\leq \kappa_{eg}\Delta$ then $m_k(x_k+s)$ is a $\kappa_{ef}, \kappa_{eg}$-fully linear model of $\phi(x+s)$ on $B(x,\Delta)$  with $\kappa_{ef}= \kappa_{eg}+\frac{L+\kappa_{hbm}}{2}$.
\end{lemma}

Thus we focus on constructing $g(x)=\nabla m({x})$.
We assume that we have access to an inexact zeroth-order oracle: $f(x)\approx {\phi}(x)$ such that 
\[
   | f(x)-{\phi}(x)| \leq \epsilon_f.
\]

Let us consider a quadratic model \eqref{eq:model_def} with $g(x)$ computed by a finite difference scheme:
 \begin{align*}
	g(x) =  \sum_{i=1}^n \frac{f(x+{\delta} u_i) - f(x)}{{\delta} } u_i, 
\end{align*}
where $u_i$ is the $i$-th column of a unitary matrix, and the Hessian approximation $H(x)$ being an arbitrary symmetric matrix, such that $\|H(x)\|_2\leq \kappa_{bhm}$. 

From analysis of the finite difference gradient approximation error (see e.g.\cite{berahas2021theoretical}) we have
\begin{align*}
	\|g({x}) - \nabla {\phi}({x})\| \leq \frac{\sqrt{n} L {\delta}}{2} + \frac{2\sqrt{n}  \epsilon_f}{{\delta} }.
\end{align*}
Thus, by selecting 
\[
{\delta  \geq  2\sqrt{\frac{\epsilon_f}{L}}}
\]
we ensure 
\[
\|g({x}) - \nabla {\phi}({x})\| \leq \sqrt{n} L\delta =\kappa_{eg}{\Delta} \quad \kappa_{eg}=\sqrt{n}L \frac{\delta }{\Delta}
\]

In conclusion,  when the finite difference scheme is used with $\delta=\Delta_k$, at each iteration $k$, then the resulting model is 
$\kappa_{ef}, \kappa_{eg}$-fully linear in $B(x_k,\Delta_k)$ with 
 $\kappa_{eg}=\sqrt{n}L$ and $\kappa_{ef}=\frac{L+\kappa_{hbm}}{2}+\sqrt{n}L$, as long as $\Delta_k\geq   2\sqrt{\frac{\epsilon_f}{L}}$. This lower bound on $\Delta_k$ is,  essentially without loss of generality,  implied by  the condition of Theorem \ref{thm:tr_complexity_ef} that  $\Delta_k\geq \sqrt{\frac{2\epsilon_f}{\tau C_2}}$ since $\sqrt{\frac{2\epsilon_f}{\tau C_2}}$ can be assumed to be larger than $2\sqrt{\frac{\epsilon_f}{L}}$ without notable impact on complexity. 

We now can apply Theorem \ref{thm:tr_complexity_ef} to bound the total number of calls to the zeroth order oracle.  We are specifically interested in the dependence of this oracle complexity on $n$ and $\epsilon$. The dependence on $\epsilon$ is explicit in the bound on $K_\epsilon$ which is ${\cal O}(\epsilon^{-2})$. 
There are several constants in the bound in Theorem \ref{thm:tr_complexity_ef}, however, only $\kappa_{eg}$ and $\kappa_{ef}$ depend on 
$n$. Other constants are not dimension dependent\footnote{For some problems $L$ can be dimension dependent but we do not consider this here.}. 
Both $\kappa_{ef}$ and $\kappa_{eg}$ scale in the same way with $n$. 
 The total number of iterations $K_\epsilon$  is bounded as ${\cal O}  \left(\frac{{\kappa_{eg}^2}}{\epsilon^2}\right )$ and 
each iteration requires $n+1$ function evaluations. Thus the total worse case oracle complexity to achieve $\|\nabla {\phi}(x_k)\| \leq \epsilon$ for any $\epsilon>  \sqrt{\frac{2\epsilon_f}{\gamma^2 \tau C_2\bar C_1^2}}$ is 
\[
N_{\epsilon}={\cal O}(n^2\epsilon^{-2}).
\]

Alternatively, choosing $\delta=\frac{\Delta_k}{\sqrt{n}}$, ensures that  $\kappa_{eg}=L$ and the total complexity reduces 
to  $ {\cal O}  \left({n}\epsilon^{-2}\right )$, but since  we also have to have $\delta  \geq  2\sqrt{\frac{\epsilon_f}{L}}$, this implies that convergence holds only for $\epsilon>2\sqrt{\frac{n\epsilon_f}{\gamma^2\bar C_1^2L}}$. We conclude that if function noise $\epsilon_f$ is appropriately small compared to the desired accuracy $\epsilon$ and a  finite difference  method is used  then choosing $\delta=  \frac{\Delta_k}{\sqrt{n}}$ is the better strategy. 

The main drawback of the finite difference approach is lack of flexibility in reusing the past function values. Thus this method, while simple to analyze,  is not as practical as the method we discuss in the next section. We will see that  the radius of sampling will be the same as the trust region radius for the reasons that will be understood when we
 describe the method.

\section{Geometry-correcting  method based on Lagrange polynomials.}
\label{sec:powell}

Before introducing the method we wish to analyze in this paper we need to discuss an important tool utilized by these algorithms - Lagrange polynomials. 
The concepts and the definitions below can be found in \cite{DFOBook}.

Lagrange polynomials and associated concepts will be defined with respect to a space of polynomials ${\cal P}$ of dimension $p$.
Typically ${\cal P}$ is either the set of linear or quadratic polynomials,  but it also can be a set of quadratic polynomials with a pre-defined Hessian sparsity pattern. 

\begin{definition}{\bf Lagrange Polynomials }

Given a space of polynomials ${\cal P}$ of dimension $p$ and a set of points ${\cal Y}=\{y_1,\ldots,y_{p}\}\subset \R^n$,
a set of $p$ polynomials $\ell_j(s)$ in ${\cal P}$ for $j=1,\ldots,p$,
is called a basis of Lagrange polynomials associated with ${\cal Y}$, if
\[
\ell_j(y_i) = \delta_{ij} = \left \{\begin{array}{c} 1
\;\;\; \mbox{if} \;\;\; i=j,\\
0 \;\;\; \mbox{if} \;\;\; i\neq j.
\end{array} \right .
\]
If the basis of Lagrange polynomials exists for the given ${\cal Y}$ 
then ${\cal Y}$ is said to be \emph{poised}. 
\end{definition}

\begin{definition}{\bf $\Lambda$--poisedness  }
Given a space of polynomials ${\cal P}$ of dimension $p$, $\Lambda > 0$, and a set ${\cal B} \subset \R^n$.
A poised set ${\cal Y} = \{ y_1,\ldots,y_{p} \}$ is said to
be $\Lambda$--poised in ${\cal B}$ if ${\cal Y}\subset {\cal B}$ and for the basis of Lagrange polynomials associated with ${\cal Y}$, it holds that
\[
\Lambda \; \geq \; \max_{j=1,\ldots, p} \max_{s\in {\cal B}} |\ell_j(s)|.
\]
\end{definition}

\begin{algorithm}[ht]
\caption{~\textbf{Geometry-correcting algorithm}}
\label{alg:powell}
{\bf Inputs:} A zeroth-order oracle $f(x)\approx {\phi}(x)$, a space of polynomials ${\cal P}$ of dimension $p$, $\Delta_0$,  $x_0$,  $\gamma \in (0,1)$ $\eta_1 >0$, $\eta_2 >0$, $\Lambda>1$.\\ 
{\bf Initialization} An initial set ${\cal Y}_0$ such that $|{\cal Y}_0| \leq p$, and the function values $f(x_0)$, $f(x_0 + y_i)$, $y_i \in {\cal Y}_0$. \\
\For{$k=0,1,2,\dots$} {
	\nl Build a quadratic model $m_k(x_k + s) $ as in \eqref{eq:model_def} using $f(x_k)$ and 
	$f(x_k + y_i)$, $y_i\in {\cal Y}_k$. 
	 \\
	\nl Compute a trial step $s_k$ and  ratio $\rho_k$  as in Algorithm \ref{alg:tr}. \\
       \nl  {\em Successful iteration:} $\rho_k\geq \eta_1$ and     $\|g_k\|\geq \eta_2\Delta_k$. \\
      Set $x_{k+1} = x_k+s_k$, $\Delta_{k+1}=\gamma^{-1}\Delta_k$.\\
 \[
{j_k^* =\arg \max_{j=1,\ldots, p} \|y_j\|.}
\]
Replace the furthest interpolation point and shift ${\cal Y}_{k+1} = ({\cal Y}_k\setminus \{y_{j_k^*}\} \cup \{0\}) - s_k$\\
\nl {\em Unsuccessful iteration:} $\rho_k < \eta_1$ or  $\|g_k\|<\eta_2\Delta_k$. Set $x_{k+1}=x_k$ and perform the first applicable steps below. 
\begin{itemize}
\item {\em Geometry correction by  adding a point:} If $|{\cal Y}_k|<p$,  ${\cal Y}_{k+1} ={\cal Y}_k \cup \{s_k\}$.
\item {\em Geometry correction by replacing a far point:} Let $j_k^* =\arg \max_{j=1,\ldots, p} \|y_j\|.$\\
 If  $\|y_{j_k^*}\|>\Delta_k$ $\Rightarrow$   {${\cal Y}_{k+1} = {\cal Y}_k\setminus \{y_{j_k^*}\} \cup \{s_k\}$ }.
\item {\em Geometry correction by replacing a "bad" point:} Otherwise, construct the set Lagrange Polynomials  
$\{\ell_j(x), i=1, \ldots, p\}$ in  ${\cal P}$ for the set ${\cal Y}_k$ and  find max  value 
\[
{(i_k^*,s_k^*) =\arg \max_{j=1,\ldots, p, s\in B(0,\Delta_k)} |\ell_j(s)|.}
\]
 If $|\ell_{i_k^*}(s_k^*)|>\Lambda$, compute  $f(x_k + s_k^*)$, ${\cal Y}_{k+1} = {\cal Y}_k\setminus \{y_{i_k^*}\} \cup \{s_k^*\}$. 
\item {\em Geometry is good:} Otherwise $\Delta_{k+1} =\gamma \Delta_k$. 
\end{itemize}
%
}
\end{algorithm}

The following assumptions will be applied to Algorithm \ref{alg:powell} on top of Assumption \ref{assum:tr}.
\begin{assumption} \label{assum:tr_powell_fl}
  There exist $\kappa_{eg}$ such that at each iteration where  ${\cal Y}_k \subseteq B(0,\Delta_k)$
 is $\Lambda$-poised,
      \[
\|g_k-\nabla \phi(x_k)\|\leq \kappa_{eg}\Delta_k
\]
and thus  $m_k$ is $\kappa_{ef},\kappa_{eg}$-fully linear in $B(x_k,\Delta_k)$ with $\kappa_{ef}=\kappa_{eg}+\frac{L+\kappa_{hbm}}{2}$
\end{assumption}  

We will show how this property follows from the properties of Lagrange polynomials later in the section after specifying a way for the model $m_k$ to be constructed. 




Lemma \ref{lem:tr_success} applies  to Algorithm \ref{alg:powell} with $\bar C_1$ defined by \eqref{eq:tr_success_ef} and Lemma \ref{lem:tr_progress}  applies 
with function reduction defined in \eqref{eq:tr_progress_ef}.

We no longer assume that $m_k$ is fully linear at each iteration, but we still have the following key lower bound on $\Delta_k$. 
\begin{lemma}[Lower bound on $\Delta_k$]\label{lem:powell_trlowerbnd}
For any $\epsilon>  \sqrt{\frac{2\epsilon_f}{\gamma^2 \tau C_2\bar C_1^2}}$, for some $\tau\in(0,1)$   let $K_\epsilon$ be the first iteration for 
which $\|\nabla \phi(x_k)\|\leq \epsilon$. Then
for all $k=1, \ldots, K_\epsilon-1$   $\Delta_k \geq \gamma \bar C_1 \epsilon$. 
\end{lemma}
\begin{proof}
From  the mechanics of the unsuccessful iteration $\Delta_k$ is decreased only when ${\cal Y}_k \in B(0,\Delta_k)$
and is $\Lambda$-poised, thus   $m_k$ is fully linear on iteration $k$. Since from Lemma \ref{lem:tr_success} 
$\Delta_k < \bar C_1 \|\nabla \phi(x_k)\|$ and fully linear $m_k$ imply a successful step, and $\|\nabla \phi(x_k)\|>\epsilon$ 
this means that $\Delta_k$ cannot be decreased below  $\gamma C_1 \epsilon$. 
\end{proof}

Let ${\cal S}_\epsilon$ be defined as in \eqref{eq:sdef} and 
      \[ {\cal U}^d_{\epsilon} := \{k\in\{0,\dots,K_\epsilon-1\}:~ \text{iteration $k$ is unsuccessful and $\Delta_k$ is decreased}\}.
    \] 

First of all, due to Lemma \ref{lem:powell_trlowerbnd}, the bound of Lemma \ref{lem:succ_iter_bnd} on the number of successful iterations holds. 
We also have the following bound. 
\begin{lemma}\label{lem:powell_unsucc_iter_bnd}
  \[
        |{\cal U}^d_\epsilon| \le  |{\cal S}_\epsilon| + \lceil{\Big(  \log_\gamma \frac{C_1\epsilon}{\Delta_0} \Big)}\rceil. 
 \]
\end{lemma}
\begin{proof}
The proof is identical to Lemma \ref{lem:unsucc_iter_bnd} with ${\cal U}^d_\epsilon$ instead of ${\cal U}_\epsilon$, since $\Delta_k$ is 
unchanged on iterations $k\in {\cal U}_\epsilon\setminus {\cal U}^d_\epsilon$. 
\end{proof}

The final challenge is to count the number of iterations in $ {\cal U}_\epsilon\setminus {\cal U}^d_\epsilon$. We call these iterations {\em  geometry-correcting} iterations, because they are designed to improve the properties of the interpolation set. Essentially, without loss of generality we assume that ${\cal Y}_k$ is either poised or can be completed to form a poised set of $p$ points. This is because ${\cal Y}_k$ can be made so by throwing away points or by arbitrary small perturbations. 
The following result bounds the number of consecutive geometry-correcting iterations. 
\begin{theorem}\label{thm:geom_correct_total_lin}
Let ${\cal P}$ be the set of linear polynomials (with dimension $p=n$). Then the number of oracle calls in consecutive unsuccessful iterations such that $k\in  {\cal U}_\epsilon\setminus {\cal U}^d_\epsilon$ is at most $3n$. 
\end{theorem}
\begin{proof}
First, after at most $n$ iterations and $n$ oracle calls,  ${\cal Y}_k$ contains $n$ points all of which  have norm at most $\Delta_k$. 
Let ${\cal Y}_k=\{y_1, \ldots, y_n\}$ and let $Y_k$ be the matrix with $i$th column $y_i\in {\cal Y}_k$, $i=1,\ldots, n$.
The set of linear Lagrange polynomials for ${\cal Y}_k$ is then defined by $\ell^k_i(s) = s^T (Y_k^{-T})_i$.
Define
\[
I_k = \{i \in [n] \mid \|y_i\| = \Delta_k, \quad y_i^T y_j = 0 \quad \forall j\in[n]\setminus \{i\} \}.
\]
That is $I_k$ is the set of points in ${\cal Y}_k$ whose norm is $\Delta_k$ and that are orthogonal to all other points in ${\cal Y}_k$. 
We claim that (i) if $|I_k| = n$ then the set is $1$--poised and thus $k$ is not geometry-correcting, i.e,  $k\not\in  {\cal U}_\epsilon\setminus {\cal U}^d_\epsilon$ and (ii) if $|I_k| < n$ and the iteration $k\in  {\cal U}_\epsilon\setminus {\cal U}^d_\epsilon$ 
  then at least one $i$ will be added to the set $|I|$ in each iteration.
The result follows from the combination of these two claims.

Suppose $|I_k| = n$, then all $y_i$ are orthogonal.
Then we have $\frac{1}{\Delta_k}Y_k$ is a matrix with orthonormal columns, thus it is an orthogonal matrix and $(\frac{1}{\Delta_k}Y_k)^{-1} = (\frac{1}{\Delta_k}Y_k)^{T}$, which implies $(\Delta_k Y_k)^{-T} = \frac{1}{\Delta_k} Y_k$.
Thus for all $i$, 
\[
\max_{s \in B(0,\Delta_k)} |\ell^k_i(s)| = \max_{s \in B(0,\Delta_k)} |s^T (Y_k^{-T})_i| = \Delta_k \| (Y_k^{-T})_i\| = \frac{1}{\Delta_k} \| (Y_k)_i\| = 1.
\]

Suppose $|I_k| < n$ and $k$ is geometry-correcting, then from the mechanism of the algorithm, the set ${\cal Y}_k$ is not $\Lambda$--poised (otherwise
$k$ would be in ${\cal U}^d_\epsilon$ or in ${\cal S}_\epsilon$).
To show that $|I_k|$ grows, we will show that in geometry-correcting iterations $I_k$ does not lose any members and that at least one index $i$ will be added to $I_k$.
%

To show that $I_k$ does not lose any of its members on geometry-correcting iterations, let $i_k^*$ be the index of the point chosen for replacement, i.e. $i_k^* = \arg \max_{i \in [n]} \max_{s \in B(0,\Delta_k)} |\ell^k_i(s)|$.
From the mechanics of the algorithm we have $\|y_{i_k^*}\|\leq \Delta_k$ and $|\ell_{i_k^*}(s_k^*)|>\Lambda$.
By permutation, we may assume that $I_k = [|I_k|]$.
We can find an orthonormal basis by extending the orthonormal set $\{ \frac{1}{\Delta_k} y_i | i \in I_k\}$.
Changing to this new basis, we can write $Y_k = \begin{bmatrix}  \Delta_k I & 0 \\ 0 & \tilde{Y_k} \end{bmatrix}$ where $I$ is the identity matrix of size $|I_k|$ and $\tilde{Y_k}$ is of size $n - |I_k|$.
We can then observe $Y_k^{-T} = \begin{bmatrix}  \frac{1}{\Delta_k} I & 0 \\ 0 & \tilde{Y_k}^{-T} \end{bmatrix}$.
It then follows that for $i \in I_k$,  $\|(Y_k^{-T})_i\|=\frac{1}{\Delta_k}$ and therefore, $\max_{s \in B(0,\Delta_k)} \ell^k_i(s) =  \Delta_k \| (Y_k^{-T})_i\| = 1$.
Thus since $ \max_{s \in B(0,\Delta_k)} |\ell^k_{i_k^*}(s)| > \Lambda \geq 1$, we cannot have $i_k^* \in I_k$.
Thus $I_k$ does not lose any of its members on geometry-correcting iterations. 

To conclude, we just need to show that at least one $i$ is added to $I_{k+1}$  compared to  $I_k$ during a geometry-correcting iteration.  
By the the mechanism of a geometry-correcting iteration of Algorithm \ref{alg:powell}, ${\cal Y}_{k+1}={\cal Y}_{k}\setminus \{y_{i_k^*}\} \cup \{s_k^*\}$ where 
\[
s_k^* =\arg \max_{x \in B(0,\Delta_k)} | s^T (Y_k^{-T})_{i_k^*}| = \Delta_k \frac{(Y_k^{-T})_{i_k^*}}{\|(Y_k^{-T})_{i_k^*}\|}. 
\]

It is easy to see that  $s_k^* := \Delta_k \frac{(Y_k^{-T})_{i^*}}{\|(Y_k^{-T})_{i^*}\|}$ satisfies (i) $\|s_k^*\| = \Delta$ and (ii) $y_j^T s_k^* = 0$ $\forall j\in[n]\setminus \{i_k^*\}$.
Since $y_{i_k^*}$ is removed from the set when $s_k^*$ is added, it follows that $I_{k+1}=I_k\cup \{i^*\}$, thus the cardinality of $I_k$ increases by at least one. 
Since only $n$ such iterations are possible consecutively, and since each geometry-correcting iteration uses at most two oracle calls, the total number of oracle calls is at most $3n$.

\end{proof}

We now can state the total  bound on the oracle complexity for Algorithm \ref{alg:powell}  in the case when ${\cal P}$ is the space of linear polynomials. This bound follows 
from the previously derived bounds  on $|{\cal S}_\epsilon|$ and $ |{\cal U}^d_\epsilon|$, the fact that each iteration in ${\cal S}_\epsilon$ and ${\cal U}^d_\epsilon$ requires only one function evaluation and Theorem \ref{thm:geom_correct_total_lin} which established that there are at most $3n$ evaluations between each iteration in ${\cal S}_\epsilon\cup {\cal U}^d_\epsilon$.  
\begin{theorem}\label{thm:tr_powell_complexity}
    Let Assumptions~\ref{assum:lip_cont} and \ref{assum:tr_powell_fl} hold. Let $\bar C_1$ and $C_2$ be defined as in \eqref{eq:tr_success_ef}  and \eqref{eq:tr_progress_ef}, respectively. For any $\epsilon>  \sqrt{\frac{2\epsilon_f}{\gamma^2 \tau C_2\bar C_1^2}}$, 
   assuming that the initial trust-region radius $\Delta_0 \ge \gamma \bar C_1 \epsilon$, then Algorithm \ref{alg:powell} achieves $\|\nabla \phi(x_k)\|\leq \epsilon$ after at most $ N_\epsilon $ function evaluations, where 
    \begin{equation} \begin{aligned} 
        N_\epsilon &=3n[|{\cal S}_\epsilon|+ |{\cal U}^d_\epsilon|]\le \frac{2(f(x_0) - f^\star)n}{(1-\tau)C_2 (\gamma \bar C_1 \epsilon)^2} + 3n\log_\gamma \frac{C_1\epsilon}{\Delta_0} \\ 
        &= \frac{4\kappa_{bhm} \Big(\max\left\{\eta_2,\ \kappa_{bhm},\ \frac{2\kappa_{ef}+4\epsilon_f}{(1-\eta_1)\kappa_{fcd}}\right\}+\kappa_{eg}\Big)^2}{(1-\tau)\gamma^2 \eta_1\eta_2\kappa_{fcd} \min\{\eta_2, {\kappa_{bhm}}\}} \cdot \frac{3n({\phi}(x_0) - {\phi}^\star)}{\epsilon^2} + 3n\log_\gamma \frac{\bar C_1\epsilon}{\Delta_0}. 
    \end{aligned} \end{equation}
\end{theorem}

\section{Ensuring fully-linear models via Lagrange polynomials.}
\label{sec:lagrange}


We now show how to ensure Assumption \ref{assum:tr_powell_fl} and derive corresponding $\kappa_{eg}$. 
For this we specify a way to construct the model $m(x)$ for Algorithm \ref{alg:powell}. 

We impose the interpolation conditions by constructing a polynomial $r(x)\in {\cal P}$ such that
\begin{equation}\label{eq:interp_cond}
r_k(y)=f(x_k+y)-f(x_k),\quad \forall y\in {\cal Y}_k. 
\end{equation}

Note that when $|{\cal Y}_k|<p$ $r_k(x)$ is not uniquely defined. There are many practically interesting alternatives of how $r(x)$ should be selected. For example, in the case when ${\cal P}$ is the space of quadratic polynomials and $|{\cal Y}_k|<\frac{n(n+1)}{2}+n$ the choices include selecting a quadratic model with the smallest Frobenius norm of the Hessian  \cite{DFOBook}, smallest Frobenius norm of the {\em change} of the Hessian \cite{MJDPowell_2004}, sparse Hessian \cite{DFOTRpaper}, etc. For the purposes of the theory we explore here, these choices matter only if it can be established that they guarantee fully linear models. We will not explore these specific choices here and will rely only on fully linear models with  $|{\cal Y}_k|=p$. We refer the reader to \cite{DFOTRpaper} for further details on $|{\cal Y}_k|<p$ case. 

We will mainly focus on the model that is constructed via linear interpolation, where $r_k(s)=g_k^Ts$ for some $g_k$
defined by \eqref{eq:interp_cond}. In this case ${\cal P}$ is a space of linear polynomials.

The model $m_k(x)$ is then defined as 
\begin{equation}\label{eq:model_def2}
m_k(x_k+s)=\phi(x_k)+r_k(s)+s^TH_ks. 
\end{equation}

The following theorem can be found in \cite{DFOBook}. 
\begin{theorem}\label{thm:lambda_to_kappaeg_book}  
Let ${\cal Y} = \{  y_1,\ldots,y_{n} \}$ be $\Lambda$--poised  in $B(0, \Delta)$.
For $r(x)$ - an affine polynomial such that $r(y)=\phi(x+y) \: \forall y \in {\cal Y}\cup \{0\}$, we have 
\[
\|g(x)-\nabla \phi(x)\| \leq \frac{L}{2} {n}\Lambda \Delta
\] 
and for all $s\in B(0, \Delta)$
\[
|\phi(x+s)-\phi(x)-r(s)|
\leq \frac{L}{2}n\Lambda \Delta^2.
\]
\end{theorem}

Thus we automatically have the following result. 
\begin{proposition} 
Assume on iteration $k$, that ${\cal Y}_k$ is 
be $\Lambda$--poised  in $B(0, \Delta_k)$ for some $\Lambda>1$ then 
 $m_k$ defined by \eqref{eq:model_def2} is $\kappa_{eg},\kappa_{ef}$-fully linear in $B(x_k,\Delta_k)$ with $\kappa_{eg}=\frac{nL\Lambda}{2}$ and
 $\kappa_{ef}=\kappa_{eg}+\frac{L+\kappa_{hbm}}{2}$.
 \end{proposition}
  Substituting these values in Theorem \ref{thm:tr_powell_complexity} and ignoring dependencies on constants, aside from $n$ and $\epsilon$ we obtain the complexity 
  \[
  N_\epsilon= {\cal O}\left (n^3\epsilon^{-2}\right ). 
  \]

 We will now show that by improving upon the analysis  in \cite{DFOBook} we are able to bring the complexity down to ${\cal O}\left (n^2\epsilon^{-2}\right )$ which is competitive with complexity of the trust region methods based on finite difference models.

\begin{theorem}\label{thm:Lambda_to_kappaeg}
Let ${\cal Y} = \{  y_1,\ldots,y_{n} \}$ be $\Lambda$--poised  in $B(0, \Delta)$.
Let $r(s) =  g^Ts$ be an affine function satisfying $r(y_i) = \phi(x+y_i)-\phi(x)$ for $i=1,\dots,n$. Then
\[
\|\nabla \phi(x) - g\| \leq \sqrt{n} L \Delta \sqrt{n (\Lambda^2 - 1) + 2}.
\]
In particular, if $\Lambda = 1 + O(\frac{1}{n})$ then we have $\|g - \nabla \phi(x)\| = O(\sqrt{n} L \Delta)$.
\end{theorem}
\begin{proof}
Let $\bar \phi(s)=\phi(x+s)-\phi(x)$, then $\bar \phi(0)=0$ and $\nabla \phi(x)=\nabla \bar \phi(0)$.
Let $Y$ be the matrix with $i$th column equal to $y_i$, let $D$ be the diagonal matrix such that $D_{ii} = \|y_i\|$, and let $\bar \phi(Y)$ be the vector with $i$th entry equal to $\bar \phi(y_i)$.
By the interpolation condition imposed on $m(x+s)$, we have $\bar \phi(Y) = Y^T g$ and thus, $g = Y^{-T}\bar \phi(Y)$.
Then we have $\nabla \phi(x) - g = \nabla \bar \phi(0) - Y^{-T}\bar \phi(Y)$ and we can bound
\[
\|\nabla \phi(x) - g\| = \|Y^{-T} D (D^{-1} Y^T \nabla \bar \phi(0) - D^{-1} \bar \phi(Y))\| \leq \sqrt{n} \|Y^{-T} D\| \| D^{-1} Y^T \nabla \bar \phi(0) - D^{-1} \bar \phi(Y)\|_{\infty}.
\]
Thus the remainder of the proof is split between bounding $\|Y^{-T} D\|$ and $\| D^{-1} Y^T \nabla \bar \phi(0) - D^{-1} \bar \phi(Y)\|_{\infty}$.

To bound $\| D^{-1} Y^T \nabla \bar \phi(0) - D^{-1} \bar \phi(Y)\|_{\infty}$, 
fix any nonzero $y \in B(0, \Delta)$ and define $h(t) =\bar  \phi(ty)$.
Observe that $h(0) = \bar \phi(0) = 0$ and $h(1) = \bar \phi(y)$.
By the fundamental theorem of calculus we have $h(1) = \int_0^1 h'(t) dt$.
We also have by the chain rule that $h'(t) = y^T \nabla \bar \phi(ty)$, thus $\bar \phi(y) = \int_0^1 y^T \nabla \bar \phi(ty) dt$.
By Lipschitzness, $\| \nabla \bar \phi(0) - \nabla \bar \phi(ty) \| \leq Lt \|y\|$.
Thus $| y^T \nabla \bar \phi(0) - y^T  \nabla \bar \phi(ty) | \leq Lt \|y\|^2$.
Thus we can bound
\[
|y^T \nabla\bar  \phi(0) - \bar \phi(y)| = \left| \int_0^1 y^T\nabla \bar \phi(0) - y^T\nabla \bar \phi(ty) dt\right| \leq \int_0^1 Lt \|y\|^2dt = \frac{1}{2} L\|y\|^2.
\]
Dividing by the norm of $\|y\|$, we have $|\frac{y^T}{\|y\|} \nabla \bar \phi(0) - \frac{\bar \phi(y)}{\|y\|}| \leq \frac{1}{2} L\|y\| \leq \frac{1}{2}L\Delta$.
Observing that for $y = y_i$, $|\frac{y^T}{\|y\|} \nabla \bar \phi(0) - \frac{\bar \phi(y)}{\|y\|}|$ is precisely the $i$th entry of the vector $D^{-1} Y^T \nabla \bar \phi(0) - D^{-1} \bar \phi(Y)$, we can conclude
\[\|D^{-1}Y^T \nabla \bar \phi(0) - D^{-1} \bar \phi(Y)\|_\infty \leq \frac{1}{2} L\Delta.\]

To bound $\|Y^{-T} D\|$,
Let $(Y^{-T})_i$ denote the $i$th column of $Y^{-T}$.
One can check that for all $i=1,\dots,n$, the Lagrange polynomials satisfy $\ell_i(s) = s^T (Y^{-T})_i$.
Therefore, the poisedness condition implies that $\max_{s \in B(0,\Delta)} \ell_i(s) = \max_{s \in B(0,\Delta)} s^T (Y^{-T})_i = \Delta\|(Y^{-T})_i\| \leq \Lambda $.
Thus for $i=1,\dots,n$, we have $\|(Y^{-T})_i\| \leq \frac{\Lambda}{\Delta}$.
Let $A = DY^{-1} Y^{-T}D$ and observe that $\text{Tr}(A) = \text{Tr}(DY^{-1} Y^{-T}D) = \sum_{i=1}^n \|(Y^{-T})_i\|^2 \|y_i\|^2 \leq n\Lambda^2$.
Also observe that  $\text{Tr}(A^{-1}) = \text{Tr}(D^{-1}Y^T YD^{-1}) = \sum_{i=1}^n  \frac{\|y_i\|^2}{\|y_i\|^2} = n$.
Denote by $\lambda_i > 0$ the eigenvalues of $A$ for $i=1,\dots,n$ and observe that $\frac{1}{\lambda_i}$ are the eigenvalues of $A^{-1}$.
Then $\sum_{i=1}^n \lambda_i + \frac{1}{\lambda_i} = \text{Tr}(A) + \text{Tr}(A^{-1}) \leq n (\Lambda^2 + 1)$.
For all positive numbers $\lambda_i$, we have $\lambda_i + \frac{1}{\lambda_i} \geq 2$.
Let $i' \in [n]$ maximize $\lambda_i + \frac{1}{\lambda_i}$.
Then
\[
\lambda_{i'} + \frac{1}{\lambda_{i'}} = \left(\sum_{i=1}^n \lambda_i + \frac{1}{\lambda_i} \right) - \left( \sum_{i\neq i'} \lambda_i + \frac{1}{\lambda_i} \right) \leq n (\Lambda^2 + 1) - 2(n-1) = n (\Lambda^2 - 1) + 2.
\]
Thus we have
\[
\|A\| = \max_{i\in[n]} \lambda_i \leq \max_{i\in[n]} \lambda_i + \frac{1}{\lambda_i} \leq n (\Lambda^2 - 1) + 2.
\]
Then observe $\|Y^{-T}D\| = \|A\|^{\frac{1}{2}} \leq \sqrt{n (\Lambda^2 - 1) + 2}$.

Combining the bounds for $\|Y^{-T} D\|$ and $\| D^{-1} Y^T \nabla \bar \phi(0) - D^{-1} \bar \phi(Y)\|_{\infty}$, we have
\[
\|\nabla \phi(x) - g\|= \|\nabla\bar  \phi(0) - g\| \leq \frac{1}{2}\sqrt{n} L \Delta \sqrt{n (\Lambda^2 - 1) + 2}.
\]
In particular, for $\Lambda = 1 + O(\frac{1}{n})$, we have $\sqrt{n (\Lambda^2 - 1) + 2} = O(1)$ and thus $\|\nabla \phi(0) - g\| = O(\sqrt{n} L \Delta)$.
\end{proof}

We thus have a result that $\Lambda$-poisedness implies that $m_k$, defined as in \eqref{eq:model_def2} and $r(s)$ defined by \eqref{eq:interp_cond} with exact zeroth order oracle, 
is $\kappa_{eg}, \kappa_{ef}$-fully linear with $\kappa_{eg}=\sqrt{n} L \Delta \sqrt{n (\Lambda^2 - 1) + 2}$ and 
$\kappa_{ef}$ as in Lemma \ref{lem:fully-lin}. 

Let us now extend the result to the inexact  zeroth order oracle  where $|f(x)-\phi(x)|\leq \epsilon_f$ for all $x$. 
 Theorem \ref{thm:Lambda_to_kappaeg} is modified as follows. 
\begin{theorem}\label{thm:Lambda_to_kappaeg_err}
Let ${\cal Y} = \{  y_1,\ldots,y_{n} \}$ be such that ${\cal Y}$
is $\Lambda$--poised  in $B(0, \Delta)$.
Let $r(s) =  g^Ts$ be an affine function satisfying $r(y_i) = f(x+y_i)-f(x)$ for $i=1,\dots,n$, with $|f(x+y)-\phi(x+y)| \leq\epsilon_f$ for $y \in {\cal Y} \cup \{0\}$. Then
\[
\|\nabla \phi(x) - g\| \leq  \sqrt{n (\Lambda^2 - 1) + 2}\left(\frac{1}{2}\sqrt{n} L \Delta +\sqrt{n} \frac{2 \epsilon_f \Lambda }{\Delta}\right).
\]
\end{theorem}
\begin{proof}
Define $\bar \phi$, $D$, and $\bar \phi(Y)$ as in the proof of Theorem \ref{thm:Lambda_to_kappaeg}.
Diverging from that proof, the interpolation condition changes from $Y^Tg=\bar \phi(Y)$ to  
\[
g^Ty_i = \phi(x+y_i)-\phi(x)+(f(x+y_i)-\phi(x+y_i))-(f(x)-\phi(x)), \quad i=1, \ldots, n
\]
thus we have $g = Y^{-T} \bar \phi(Y)+Y^{-T} E$, where $E$ is a vector with components $(f(x+y_i)-\phi(x+y_i))-(f(x)-\phi(x))$.
Then we can bound error
\begin{align*}
\|\nabla \phi(x) - g\| &= \|Y^{-T} D (D^{-1} Y^T \nabla \bar \phi(0) - D^{-1} \bar \phi(Y) - D^{-1} E)\| \\
&\leq \sqrt{n} \|Y^{-T} D\| \| D^{-1} Y^T \nabla \bar \phi(0) - D^{-1} \bar \phi(Y) - D^{-1} E\|_{\infty} \\
&\leq \sqrt{n} \|Y^{-T} D\| \left( \| D^{-1} Y^T \nabla \bar \phi(0) - D^{-1} \bar \phi(Y)\|_{\infty} + \| D^{-1} E\|_{\infty} \right) \\
&\leq \sqrt{n} \|Y^{-T} D\| \left( \| D^{-1} Y^T \nabla \bar \phi(0) - D^{-1} \bar \phi(Y)\|_{\infty} + \| D^{-1}\|_{\infty} \| E\|_{\infty} \right).
\end{align*}
We can bound $\|Y^{-T} D\|$ and $\| D^{-1} Y^T \nabla \bar \phi(0) - D^{-1} \bar \phi(Y)\|_{\infty}$ identically as in the proof of Theorem \ref{thm:Lambda_to_kappaeg}.
To bound $\| E\|_{\infty}$ observe that the condition that $|f(x+y)-\phi(x+y)| \leq\epsilon_f$ for $y \in {\cal Y} \cup \{0\}$ implies $\| E\|_{\infty} \leq 2 \epsilon_f$.
To bound $\| D^{-1}\|_{\infty}$ recall that $D_{ii} = \|y_i\|$.
By the properties of Lagrange polynomials we have $\ell_i(y_i) = 1$.
By linearity, we have $\ell_i(\Delta \frac{y_i}{\|y_i\|}) = \frac{\Delta}{\|y_i\|}$.
By the poisedness condition we have $\ell_i(\Delta \frac{y_i}{\|y_i\|}) \leq \Lambda$.
Combining these bounds we have $\frac{\Delta}{\|y_i\|} \leq \Lambda$ and thus $\frac{1}{\|y_i\|} \leq \frac{\Lambda}{\Delta}$.
Thus $\| D^{-1}\|_{\infty} \leq  \frac{\Lambda}{\Delta}$.
The result follows.
\end{proof}

 It follows, that if ${\cal Y}_k$ is $\Lambda$-poised in $B(0, \Delta_k)$ with  $\Delta_k\geq 2\sqrt{\frac{\Lambda\epsilon_f}{L}}$ we have $\|\nabla \phi(x_k) - g_k\| \leq  \sqrt{n (\Lambda^2 - 1) + 2}(\sqrt{n} L \Delta_k)$. We thus have a result that $\Lambda$-poisedness, with $\Delta_k\geq 2\sqrt{\frac{\Lambda\epsilon_f}{L}}$, implies that $m_k$ is $\kappa_{eg}, \kappa_{ef}$-fully linear with $\kappa_{eg}=\sqrt{n} L \Delta \sqrt{n (\Lambda^2 - 1) + 2}$ and 
$\kappa_{ef}$ defined by Lemma \ref{lem:fully-lin}. 


Without loss of generality we can assume that $L\geq \frac{\Lambda C_2}{2}$ in Theorem \ref{thm:Lambda_to_kappaeg_err}. This implies that for any $\epsilon>  \sqrt{\frac{2\epsilon_f}{\gamma^2 \tau C_2\bar C_1^2}}$, $\Delta_k \ge \gamma \bar C_1 \epsilon$ (for some $\tau\in(0,1)$) implies $\Delta_k\geq 2\sqrt{\frac{\Lambda\epsilon_f}{L}}$. 
Then the  immediate consequence of Theorem \ref{thm:tr_powell_complexity} is as follows

\begin{theorem}\label{thm:tr_powell_complexity_final}
    Let Assumptions~\ref{assum:lip_cont} and \ref{assum:tr_powell_fl} hold. For any $\epsilon>  \sqrt{\frac{2\epsilon_f}{\gamma^2 \tau C_2\bar C_1^2}}$, $\Delta_k \ge \gamma \bar C_1 \epsilon$ (for some $\tau\in(0,1)$), 
   assuming that the initial trust-region radius $\Delta_0 \ge \gamma \bar C_1 \epsilon$ and $\Lambda$ is chosen as $1 + {\cal O}(\frac{1}{n})$, then Algorithm \ref{alg:powell} achieves $\|\nabla \phi(x_k)\|\leq \epsilon$ after at most $ N_\epsilon $ function evaluations, where 
    \begin{equation} \begin{aligned} 
        N_\epsilon &={\cal O}(n^2\epsilon^{-2}),
            \end{aligned} \end{equation}
            with ${\cal O}$ containing additive logarithmic factors and constant that are independent of $n$ and $\epsilon$. 
\end{theorem}
This result is important because it shows than the worst case complexity of the geometry-correcting method matches that of a method based on finite differences. 
Thus nothing is lost by employing Algorithm \ref{alg:powell}, aside from additional linear algebra costs of maintaining the Lagrange polynomials. 

The following theorem demonstrates that setting $\Lambda =1 + {\cal O}(\frac{1}{n})$ is critical to our final result. In other words, the overall bound on $\kappa_{eg}$ in terms of $\Lambda$ cannot be improved. 

\begin{theorem}
For all $n\geq 2$, $L,\Delta > 0$, and $\Lambda > 1$, there is a $L$-smooth function $\phi$, and a set ${\cal Y} = \{  y_1,\ldots,y_{n} \}$  $\Lambda$--poised in $B(0, \Delta)$ such that for the  linear interpolating function $r(s) =  g^Ts$  satisfying $r(y_i) = \phi(x+y_i)-\phi(x)$ for $i=1,\dots,n$, we have $\|g - \nabla \phi(x)\| \geq O(L n \Lambda \Delta)$.
\end{theorem}
\begin{proof}
By scaling we can take $L = \Delta = 1$.
Let $x = 0$ and $\bar \phi(u)=\phi(x+u) = \frac{1}{2} u^T u$.
Let $\varepsilon$ be the unique number in the interval $(0,\frac{1}{n-1})$ such that 
\[
\frac{1}{1+\varepsilon} + \frac{\varepsilon}{(1+\varepsilon)^2(1-n\varepsilon/(1+\varepsilon))}= \Lambda^2.
\]
Define $A = (1+\varepsilon A) - \varepsilon \mathbf{1}\mathbf{1}^T$.
Let $U = \sqrt{A}$ and let $y_i=u_i$ - the $i$-th column of $U$. 
Observe for $i=1,\dots,n$ that $\|u_i\| = \sqrt{A_{ii}} = 1$ and thus, $u_i \in B(0,\Delta)$.
By the Sherman-Morrison formula, we have $A^{-1} = \frac{1}{1+\varepsilon} I + ( \frac{\varepsilon}{(1+\varepsilon)^2(1-n\varepsilon/(1+\varepsilon))})\mathbf{1}\mathbf{1}^T$.
Observe that $\max_{s \in B(0,\Delta)} \ell_i(s) = \max_{s \in B(0,\Delta)} x^T (U^{-T})_i = \|(U^{-T})_i\|$.
Since $U = \sqrt{A}$, we have $U^{-1} = \sqrt{A^{-1}}$, thus $\|(U^{-T})_i\| = \sqrt{(A^{-1})_{ii}} = \Lambda$.
Thus the set is indeed $\Lambda$--poised.
Then we have
\[
\| g - \nabla \phi(x)\| = \|g\| = \| U^{-T} \phi(U) \| = \frac{1}{2} \| U^{-T} \mathbf{1} \| = \frac{1}{2} \sqrt{\mathbf{1}^T A^{-1} \mathbf{1}} = \frac{1}{2} \sqrt{ \frac{n}{1+\varepsilon} + \frac{n^2\varepsilon}{(1+\varepsilon)^2(1-n\varepsilon/(1+\varepsilon))}}.
\]
Observe that
\[
\frac{n}{1+\varepsilon} + \frac{n^2\varepsilon}{(1+\varepsilon)^2(1-n\varepsilon/(1+\varepsilon))} = n (n(\Lambda^2 - \frac{1}{1+\varepsilon}) + \frac{1}{1+\varepsilon}).
\]
Simplifying and bounding we obtain
\[
n(\Lambda^2 - \frac{1}{1+\varepsilon}) + \frac{1}{1+\varepsilon} = n\Lambda^2 - \frac{n-1}{1+\varepsilon} \geq n \Lambda^2 - (n-1) = n(\Lambda^2-1) + 1.
\]
Combining, we have
\[
\| g - \nabla \phi(x)\| \geq \frac{1}{2} \sqrt{n}\sqrt{n(\Lambda^2-1) + 1}.
\]

\end{proof}

\section{Extensions to arbitrary polynomial basis.}
\label{sec:powell2}
Below we present the bound on the number of consecutive geometry-correcting iterations for the case of a general space of polynomials $\mathcal{P}$ defined by some  basis
$\pi(x)=\{\pi_1(x), \ldots, \pi_p(x)\}$. We overload the notation and use $\pi(x)$, for a fixed $x$, to  denote the corresponding  vector. 

\begin{theorem}\label{thm:plogp}
For a general space of polynomials $\mathcal{P}$ with dimension $p$, 
then the number of oracle calls in consecutive unsuccessful iterations such that $k\in  {\cal U}_\epsilon\setminus {\cal U}^d_\epsilon$ is $O(p \log p)$. 
\end{theorem}
\begin{proof}
First, after at most $n$ iterations and $n$ oracle calls, all points in ${\cal Y}_k$ will have norm at most $\Delta_k$.
Let $\text{vol}(\pi(Y))$ denote the volume of the simplex with vertices $\pi(y),\ y\in Y$. Let $Y_i(s)$ equal $Y$ with point $y_i$ replaced by $s$. It follows that
     \[
     |\ell_i(s)|=\frac{\text{vol}(\pi(Y_i(s)))}{\text{vol}(\pi(Y))}.
     \] 
In other words, replacing  $y_i$ with $s$ increases the volume by a factor of $|\ell_i(s)|$.

We claim that if the vectors in $Y$ are $\Lambda$--poised, then
\[
\frac{\text{vol}(\pi(Y))}{\text{vol}(\pi(Y^*))} \geq (\sqrt{p} \Lambda)^{-p}
\]
where $Y^*$ is the matrix with columns consisting of vectors $y_i^*,\dots,y_p^* \in \mathcal{B}$ which maximize $\text{vol}(\pi(Y^*))$.
Then we have
\[
\frac{\text{vol}(\pi(Y^*))}{\text{vol}(\pi(Y))} = \left| \frac{\det(\pi(Y^*))}{\det(\pi(Y))} \right| = \left| \det(\pi(Y)^{-1} \pi(Y^*)) \right|. 
\]
One can check that $(\pi(Y)^{-1} \pi(Y^*))_{ij} = \ell_i(y_j^*)$.
Thus, we can conclude the proof of the claim via Hadamard's inequality:
\[
\left| \det(\pi(Y)^{-1} \pi(Y^*)) \right| \leq \prod_{i}^{p} \sqrt{\sum_{j}^p |\ell_i(y_j^*)|^2} \leq (\sqrt{p} \Lambda)^p.
\]

Let $V(t) = ((\sqrt{p} t))^{-p}$ and let $N(t)$ be the number of consecutive geometry correction steps until $\text{vol}(\pi(Y))/\text{vol}(\pi(Y^*)) \geq t$.

By the argument in part 2 of the proof of Theorem 6.3 of \cite{DFOBook}, after at most $p$ iterations, we have a $2^p$--poised set, so we have $N(V(2^p)) \leq p$.
Now for any $k = 1,\dots,p-1$ let us bound $N(V(2^k))$.
Suppose we take $N(V(2^{k+1}))$ steps and have $\text{vol}(\pi(Y))/\text{vol}(\pi(Y^*)) \geq V(2^{k+1})$.
In every subsequent step, either we achieve $2^k$--poisedness, and thus $\text{vol}(\pi(Y))/\text{vol}(\pi(Y^*)) \geq V(2^k)$ by the above claim, or poisedness is greater than $2^k$ and we improve volume by a factor of at least $2^k$. 
Since we only need to improve volume by a factor of at most $\frac{V(2^k)}{V(2^{k+1})} = 2^p$, we need to take at most $\lceil \log_{2^k}(2^p)\rceil = \lceil\frac{p}{k} \rceil$ subsequent steps.
Thus we have $N(V(2^k)) \leq N(V(2^{k+1})) + \lceil\frac{p}{k} \rceil$.

Now, by induction we have 
\[
N(V(2)) \leq p + \sum_{k=1}^{p-1} \lceil\frac{p}{k} \rceil \leq 2p + \sum_{k=1}^{p-1} \frac{p}{k} = 2p + p \left( \sum_{k=1}^{p-1} \frac{1}{k} \right) \leq 3p + p \log(p).
\]

To conclude, for any additional steps we may assume that poisedness is greater that $\Lambda$ and thus, volume will increase by a factor of at least $\Lambda$.
However, the ratio of volume to the maximum volume cannot grow to be larger than $1$.
Thus, we have the total number of steps is at most
\[
\log_{\Lambda} \left( \frac{1}{V(2)} \right) + N(V(2)) \leq \frac{1}{\log{\Lambda}} \left( p\log( \sqrt{p} 2) \right) + 3p + p \log(p) = O(p \log(p)).
\]
Since each geometry correction iteration uses at most two oracle calls, the result follows.
\end{proof}

To have a complexity bound for Algorithm \ref{alg:powell} based on general space of polynomials $\mathcal{P}$ we need to establish
how $\Lambda$-poisedness of ${\cal Y}_k$ implies fully linear model $m_k$ and derive the expression for $\kappa_{eg}$. Certain results 
for this were cited or established in  \cite{DFOBook}. As in the linear case those bounds may not be optimal and to develop better bounds 
one would likely need to specify $\mathcal{P}$. We leave this for future research, noting that for now, the bounds in  \cite{DFOBook} show that
$\Lambda$-poised ${\cal Y}_k$ implies fully linear model $m_k$ with  $\kappa_{eg}$ having a polynomial dependence on $n$. 
In the next section we show how such models can be used in low dimensional subspaces, where  $\kappa_{eg}$ will depend on the subspace dimension and thus
 the exact nature of this dependence has small impact on the total complexity.


\section{Model based trust region methods in subspaces.}
\label{sec:subspace_ffd}
We now consider a trust region method where a model $m(x)$ is built and optimized in a random low-dimensional subspace of ${\mathbb R}^n$. The idea of using  random subspace embeddings 
within derivative free methods has gained a lot of popularity in the literature lately. Specifically, in \cite{Cartis2023} random embeddings based on  sketching matrices and 
Johnson-Lindenstrauss Lemma, are used together with a model based trust region method. The trust region method involved is different from what we discuss here and has the
 same worst case oracle complexity complexity ${\cal O}(\frac{n^2}{\epsilon^2})$ as our full-space method. 
Here we will show that combining our analysis in the paper with the use of random projections
 results  in better complexity in terms of the dependence on $n$ - ${\cal O}(\frac{n}{\epsilon^2})$.  
In \cite{Kozak2021} random projections are used together with finite difference subspace gradient approximation and gradient descent. 
The resulting oracle complexity of this subspace method is the same as that of full-space finite difference gradient descent method ${\cal O}(\frac{n}{\epsilon^2})$. 
The purpose of this section is to show that unlike finite difference gradient descent, the trust region interpolation based  methods, 
such as the geometry-correcting Algorithm \ref{alg:powell} improve their complexity  in terms of the dependence on $n$ when used in subspaces versus their full space versions. 

In this section we will assume that the zeroth oracle is exact. Extension to inexact oracle is subject of future study since it requires certain extension of existing theory of randomized trust region methods that are beyond the subject of this paper. We will elaborate on this at the end of this section.

Given a matrix $Q\in {\mathbb R}^{n\times q}$, with $q\leq n$ and orthonormal columns, $QQ^T\nabla \phi(x)$ is an orthogonal projection of $\nabla \phi(x)$ onto a subspace spanned by the columns of $Q$ (we will call it a subspace induced by $Q$).  We also define a reduction of  $\phi(x)$  to the subspace, given by $Q$ around $x$: ${\hat \phi}(v)={\phi}(x+Qv)$, $v\in {\mathbb R}^{q}$, which implies $Q\nabla {\hat \phi}(0)=QQ^T\nabla {\phi}(x)$. Similarly we define 
${\hat m}(v)={m}(x+Qv)$,  $v\in {\mathbb R}^{q}$, which implies $Q\nabla {\hat m}(0)=QQ^T \nabla {m}(x)$. 

We now present a modified trust-region algorithm that constructs models and computes steps in the subspace. 
At each iteration  $k \in \{0,1,\dots\}$ the algorithm  chooses  $Q_k\in {\mathbb R}^{n\times q}$ with orthonormal columns. The model $m_k$ is defined as 
\begin{equation}\label{eq:model_def_sub}
	m_k(x_k+Q_kv) = {\phi}(x_k) + g_k^TQ_kv  + \frac{1}{2} v^TQ_k^T H_k Q_kv. 
\end{equation}
For any vector $v$, $ g_k^T Q_kv = g_k^TQ_kQ_k^Tg_k^T Q_kv $ thus without loss of generality, we will assume that $Q_kQ_k^Tg_k=g_k$, in other words, $g_k$ lies in the 
subspace induced by  $Q_k$. 
We define the trust region in the subspace induced by $Q_k$ as  $B_{Q_k}(x_k, \Delta_k)=\{z:\, z=x_k+Q_kv, \ \|v\|\leq \Delta_k\}$.


\begin{algorithm}[ht] 
    \caption{~\textbf{Trust region method based on fully-linear models in subspace}}
    \label{alg:tr_sub}
       {\bf Inputs:} Exact zeroth order oracle $f(x)=\phi(x)$, initial  $x_0$, $\Delta_0$, and  $\eta_1\in(0,1)$, $\eta_2 > 0$, and $\gamma\in(0,1)$.  \\
      \For{$k=0,1,2,\cdots$}{
           \nl  Choose $Q_k\in {\mathbb R}^{n\times q}$ with orthonormal columns.  Compute model $m_k$ as in \eqref{eq:model_def_sub}. \\
        \nl Compute a trial step $x_k+s_k$  where $s_k=Q_kv_k$ with $v_k\approx \arg\min_v \{m_k(x_k+Q_kv):~\|v\|\leq \Delta_k\}$.\\
        \nl Compute the ratio $\rho_k$ as in Algorithm \ref{alg:tr}. \\
        \nl Update the iterate and the TR radius as in Algorithm \ref{alg:tr}. 
      }
\end{algorithm}

We will assume here that  Assumption \ref{assum:tr} holds. 
We will also need  model $m_k$ to be fully linear but only with respect to the subspace. 

\begin{definition}[Fully-linear model in a subspace] \label{def:fully-linear-subspace}
Given a matrix with orthonormal columns $Q\in {\mathbb R}^{n\times q}$, let  $B_Q(x, \Delta)=\{z:\, z=x+Qv, \ \|v\|\leq \Delta\}$. 
Let
 \begin{equation}\label{eq:model_def_sub_nok}
	m(x+Qv) = {\phi}(x) + g^TQv  + \frac{1}{2} v^TQ^T H Qv
\end{equation}
and ${\hat m}(v)={m}(x+Qv)$,  $v\in {\mathbb R}^{q}$. 
We say that model $m(x+s)$ is $\kappa_{ef}, \kappa_{eg}$-fully linear model of $\phi(x+s)$  on 
 $B_Q(x, \Delta)$ if 
\[
\|\nabla \hat m(0)-\nabla \hat \phi (0)\|\leq \kappa_{eg}\Delta
\]
and 
\[
|\hat m(v)-\hat \phi (v)|\leq \kappa_{ef}\Delta^2
\]
for all $\|v\|\leq \Delta$. 
\end{definition}

\begin{definition}[Well aligned subspace] \label{def:well-aligned-subspace}
The subspace spanned by columns of $Q$ is $\kappa_g$-well aligned with  $\nabla \phi(x)$ for a given  $x$ if
      \begin{equation}\label{eq:subspace_req_leq}
\|QQ^T\nabla \phi(x)-\nabla \phi(x) \|\leq \kappa_g\|\nabla \phi(x) \|
\end{equation}    
for some $\kappa_g\in [0, 1)$. 
\end{definition}

Condition \eqref{eq:subspace_req} is on the properties of the subspace. Essentially, it is required that the cosine of the angle between the gradient and its projection onto the subspace induced by $Q$ is not too small. Similar conditions (and similar terminology) have been used in \cite{Cartis2023}.  

We will discuss later how this requirement can be satisfied with sufficiently high probability by randomly generated subspaces. 
The following lemma is a simple consequence of conditions above.  

\begin{lemma}\label{lem:grad_error_sub}
On iteration $k$, $Q_k$  is $\kappa_g$-well aligned with  $\nabla \phi(x_k)$ if and only if 
  \begin{equation}\label{eq:subspace_req}
\|Q_kQ_k^T\nabla \phi(x_k) \|^2\geq (1-\kappa_{g}^2)\|\nabla \phi(x_k) \|^2
\end{equation} 
 If $m(x_k+s)$ is $\kappa_{ef}, \kappa_{eg}$-fully linear model of $\phi(x_k+s)$ on  $B_{Q_k}(x_k, \Delta_k)$ then 
  \begin{equation}\label{eq:grad_error_sub}
 \|g_k-Q_kQ_k^T\nabla\phi(x_k)\|\leq \kappa_{eg}\Delta_k 
 \end{equation}
 \end{lemma}
 \begin{proof}
 The first statement easily follows from the fact that $Q_kQ_k^T$ is an orthogonal projection.  From this and the fact that 
 $g_k=Q_kQ_k^Tg_k$ we have
\[
 \|g_k-Q_kQ_k^T\nabla\phi(x_k)\|=  \|Q_k\nabla \hat m(0)-Q_k\nabla\hat \phi(0)\| = \|\nabla \hat m(0)-\nabla\hat \phi(0)\|
 \]
 Thus  the second statement follows from the fully linear assumption. 
 
 \end{proof}

We now show how the analysis of  Algorithm \ref{alg:tr} easily extends to Algorithm \ref{alg:tr_sub} under appropriate assumptions on the models and the subspaces. 
\begin{lemma}[sufficient condition for a successful step] \label{lem:tr_sub_success}
  Under Assumptions~\ref{assum:lip_cont} and \ref{assum:tr}, if $Q_k$ is $\kappa_g$-well aligned with $\nabla \phi(x_k)$, $m_k(x_k+s)$ is a $\kappa_{ef}, \kappa_{eg}$-fully linear model of $\phi(x_k+s)$ on $B_{Q_k}(x_k,\Delta_k)$  and if 
    \begin{equation} \label{eq:tr_sub_success_1}
        \Delta_k \le \sqrt{1-\kappa_g^2} C_1 \|\nabla \phi(x_k)\| 
    \end{equation}
   where
   \[
     C_1 = (\max\left\{\eta_2,\ \kappa_{bhm},\ \frac{2\kappa_{ef}}{(1-\eta_1)\kappa_{fcd}}\right\}+\kappa_{eg})^{-1}
    \]
    then $\rho \ge \eta_1$, $\|g_k\| \ge \eta_2 \Delta_k$, and $x_{k+1} = x_k+s_k$, i.e. the iteration $k$ is successful. 
\end{lemma}

\begin{proof}
    Due to Lemma \ref{lem:grad_error_sub}, specifically \eqref{eq:grad_error_sub}   by triangle inequality, 
    \[ \|Q_kQ_k^T\nabla \phi(x_k)\| \le \|g_k\| + \kappa_{eg}\Delta_k
    \]
    and also due to Lemma \ref{lem:grad_error_sub}
    \begin{equation}\label{eq:tr_sub_success} 
      \Delta_k \le \sqrt{1-\kappa_g^2}C_1 \|\nabla \phi(x_k)\| \leq C_1 \|Q_kQ_k^T\nabla \phi(x_k)\|
    \end{equation}
      
    By \eqref{eq:tr_sub_success} we have  
    \[ \begin{aligned}
       (\max\{\kappa_{bhm},\eta_2, \frac{2\kappa_{ef}}{(1-\eta_1)\kappa_{fcd}}\} + \kappa_{eg}) \Delta_k & \le 
         (\|g_k\| + \kappa_{eg}\Delta_k) 
    \end{aligned} \] 
    which implies
    \[
           \max\{\kappa_{bhm},\eta_2\} \Delta_k \le \|g_k\|. 
           \]
    This establishes that $\|g_k\| \ge \eta_2\Delta_k$ and also $m_k(x_k) - m_k(x_k+s_k) \ge \kappa_{fcd} \|g_k\| \Delta_k / 2$ by Assumption~\ref{assum:tr}. 
    Then, using  the fact that $f(x)=\phi(x)$ and the fully linear assumption on  $m_k$ in the subspace induced by $Q_k$ and recalling that $s_k=Q_kv_k$ we have 
    \[ \begin{aligned}
        \rho_k &= \frac{m(x_k) - m(x_k+s_k) + (f(x_k) - m(x_k)) - (f(x_k+s_k) - m(x_k+s_k))}{m(x_k) - m(x_k+s_k)} \\
         &= \frac{m(x_k) - m(x_k+s_k) + (\phi(x_k) - m(x_k)) - (\phi(x_k+s_k) - m(x_k+s_k))}{m(x_k) - m(x_k+s_k)}\\
        &\ge 1 - \frac{\kappa_{ef} \Delta_k^2}{m(x_k) - m(x_k+s_k)}   \ge 1 - \frac{\kappa_{ef} \Delta_k^2}{\kappa_{fcd} \|g_k\| \Delta_k / 2} \ge 1 - \frac{2\kappa_{ef}\Delta_k}{\kappa_{fcd} (( \|Q_kQ_k^T\nabla \phi(x_k)\|- \kappa_{eg}\Delta_k)} \ge \eta_1, 
    \end{aligned} \]
    where the last step is true because $ \|Q_kQ_k^T\nabla \phi(x_k)\| \ge \big(\frac{2\kappa_{ef}}{(1-\eta_1)\kappa_{fcd}} + \kappa_{eg}\big) \Delta_k$  follows from  \eqref{eq:tr_sub_success}. 
\end{proof}

The rest of the analysis is identical to the analysis of Algorithm \ref{alg:tr} and Theorem \ref{thm:tr_complexity} holds with the same bound but slightly differently defined $C_1$, thus we have the following complexity result, 
\begin{theorem}
    Let Assumptions~\ref{assum:lip_cont} and \ref{assum:tr} hold. Let $K_\epsilon$ be the first iteration of  Algorithm \ref{alg:tr_sub} that achieves $\|\nabla \phi(x_k)\|\leq \epsilon$. 
     Assume on all iterations $k=0, \ldots, K_\epsilon-1$ $Q_k$ is $\kappa_g$-well aligned with $\nabla \phi(x_k)$ and $m_k(x_k+s)$ is a $\kappa_{ef}, \kappa_{eg}$-fully linear model of $\phi(x_k+s)$ on $B_{Q_k}(x_k,\Delta_k)$. 
   Then assuming that the initial trust-region radius $\Delta_0 \ge \gamma \hat C_1 \epsilon$, we have the bound 
    \begin{equation} \begin{aligned} 
        K_\epsilon &=|{\cal S}_\epsilon|+ |{\cal U}_\epsilon|\le \frac{2(f(x_0) - f^\star)}{C_2 (\gamma \hat C_1 \epsilon)^2} + \log_\gamma \frac{\hat C_1 \epsilon}{\Delta_0} \quad C_2 = \frac{\eta_1\eta_2\kappa_{fcd}}{2\kappa_{bhm}} \min\{\eta_2, {\kappa_{bhm}}\}; 
     \end{aligned} \end{equation}
     where
   \[
   \hat  C_1  = \frac{\sqrt{1-\kappa_{g}^2}}{\max\left\{\eta_2,\ \kappa_{bhm},\ \frac{2\kappa_{ef}}{(1-\eta_1)\kappa_{fcd}}\right\}+\kappa_{eg}}. 
    \]
\end{theorem}
The above result is not very useful by itself without the expressions for the bounds  $\kappa_{eg}$, $\kappa_{ef}$ (which we already know how to derive) and more critically, $\kappa_g$ - which we did not discuss yet. In fact we will only be able to guarantee the bound $\kappa_g$ that holds  with some probability when $Q_k$ is random. 
We discuss this bound below and then extend the analysis of Algorithm \ref{alg:tr_sub} to this case. 

\subsection{Building fully-linear models in a subspace.} 
Let us apply standard forward finite difference method to approximate the "reduced" gradient $\nabla {\hat \phi}(0)$. 
\begin{equation}\label{eq:subspace_ffd}
\hat g(0)=\sum_{i=1}^q \frac{f(x+\delta Qu_i)-f(x)}{\delta}u_i 
\end{equation}
where $u_i$ is the $i$th column of a unitary ${q\times q}$ matrix. And let us define $g(x)=Q\hat g(0)$. Since we consider an exact  oracle   $f(x)= {\phi}(x)$ we simply 
have \begin{align*}
\|\hat g(0)-\nabla {\hat \phi}(0)\|& \leq  \frac{\sqrt{q}L{\delta}}{2}
\end{align*}
Thus assuming $\delta\leq \Delta$, for $m(x)$ defined by \eqref{eq:model_def_sub}  with $g(x)=Q\hat g(0)$ we have 
$$
\|\nabla \hat m(0)-\nabla \hat \phi (0)\|\leq \kappa_{eg}\Delta
$$
with $\kappa_{eg}=\frac{\sqrt{q}L}{2}$.  Analogously to Lemma \ref{lem:fully-lin} we have the following. 

\begin{lemma}[Fully-linear in subspace models]\label{lem:fully_lin_sub}
For $m(x)$ defined by \eqref{eq:model_def_sub_nok}  such that 
$$
\|\nabla \hat m(0)-\nabla \hat \phi (0)\|\leq \kappa_{eg}\Delta
$$
  $m_k(x+s)$ is a $\kappa_{ef}, \kappa_{eg}$-fully linear model of $\phi(x+s)$ on $B(x,\Delta)$  in a subspace induced by $Q$, with 
\[
\kappa_{ef}= \kappa_{eg}+\frac{L_Q+\kappa_{hbm}}{2}
\] 
where $L_Q$ is the Lipschitz constant of $QQ^T\nabla \phi (x)$. 
\end{lemma}

Now we discuss how to generate a well aligned subspace. This subject has been studied extensively in the literature and in particular in derivative free optimization context \cite{Cartis2023,Kozak2021} and the key approach is to generate random subspaces. In \cite{Cartis2023} the sketching and Johnson-Lindenstrauss type embeddings are used and the resulting complexity is worse than what we derive here\footnote{It remains to be seen if those results can be improved by combining our new bounds related to $\Lambda$-poised models with the use of JL embeddings.}. In \cite{Kozak2021} $Q_k$ matrices are generated from the Haar distribution (random matrix with orthonormal columns) and are used in the context of gradient descent methods based on finite difference gradient approximation. Here we also use Haar distribution and rely on the following result. 

\begin{lemma}\label{lem:haar_bound}
              For $q \geq 3$, $Q \in {\mathbb R}^{n \times q}$ drawn from the Haar distribution of the set of matrices with orthonormal columns, 
              and any nonzero vector $v \in {\mathbb R}^n$, we have

              \[
              {\mathbb P}\left [\|QQ^Tv\|^2\geq ( \frac{q}{10n})\| v \|^2 \right ] \geq \frac{243}{443} > \frac{1}{2}.
              \]

\end{lemma}

\begin{proof}

              By  Lemma~1 in \cite{Kozak2021}  we have

              \[
              \frac{\|QQ^Tv\|^2}{\| v \|^2} \sim \operatorname{Beta}\!\left(\frac{q}{2},\,\frac{n-q}{2}\right).         
                 \]

              Let \(X \sim \operatorname{Beta}(q/2,(n-q)/2)\).  We have \( \mathbb{E}[X] = \frac{q}{n}\) and \(\mathbb{V}[X] = \frac{q(n-q)}{n^2(n/2+1)} \leq \frac{2q}{n^2} \leq \frac{2q^2}{3n^2}\).  Thus by the Paley-Zygmund inequality we have

              \[
              {\mathbb P}\left [\|QQ^Tv\|^2\geq \left( \frac{q}{10n}\right )\| v \|^2 \right ] = {\mathbb P}(X \geq (1/10)\mathbb{E}[X]) \geq \frac{(1-1/10)^2\mathbb{E}[X]^2}{\mathbb{V}[X] + (1-1/10)^2\mathbb{E}[X]^2} \geq \frac{(1-1/10)^2\frac{q^2}{n^2}}{\frac{2q^2}{3n^2} + (1-1/10)^2\frac{q^2}{n^2}} = \frac{243}{443}.
             \]

\end{proof}

The immediate conclusion is that on each iteration $k$, $Q_k$ drawn from the Haar distribution (with $q>3$ is $\kappa_g$-well aligned with probability  $\theta\geq 
\frac{243}{443}>1/2$ with $\kappa_g=\sqrt{1-\frac{q}{10n}}$. 

\subsection{Complexity analysis under random subspace selection.} 
If at each iteration $k$ $Q_k$ is chosen randomly,  Algorithms~\ref{alg:tr_sub} can be viewed as stochastic process $\{Q_k, x_k, m_k, \Delta_k, s_k\}$. All quantities computed by the algorithm are random variables, with some abuse of notation we will denote these quantities by the same letters as their realizations. It should be  clear from the context which one we refer to.

Let $\mathcal{F}_{k-1}$ denote the $\sigma$-algebra generated by the first $k-1$ iterations, $\mathcal{F}_{k-1}=\sigma\left( Q_0, Q_1, \ldots Q_{k-1}\right )$. 
We note that the random variables $x_k$ and $\Delta_k$ measurable with respect to  $\mathcal{F}_{k-1}$, the random variables $m_k$, $s_k$ and $\rho_k$ are measurable with respect to  $\mathcal{F}_{k}$. 
The random variable  $K_\epsilon =\min\{ k:\ \|\nabla \phi(x_k)\|\leq \epsilon\}$ is a stopping time adapted to  the filtration $\{\mathcal{F}_{k-1}\}$.  

Note that Lemmas \ref{eq:tr_sub_success_1} and \ref{eq:tr_progress} hold for each realization of the algorithm. However, the difficulty in carrying out the analysis lies in the fact that
$\Delta_k$ is no longer bounded from below on all iterations, that is Lemma \ref{lem:delta_bnd} does not hold, thus the function improvement provided by Lemma \ref{eq:tr_progress} by itself does not imply the bound on the number of successful iterations. Following the analysis in \cite{cartis2018global} (and other papers such as \cite{gratton2018complexity}) we will consider different types of iterations and derive common bounds. 

First we will define several additional random variables. Let 
\begin{align*}
I_k&=\mathbbm{1}\{Q_k\ \text{ is\ } \kappa_g \text{-well\ aligned\ with\ }  \nabla \phi(x_k)\},\\
A_k&=\mathbbm{1}\{{\rm Iteration\ } k\ {\rm is\ successful\ i.e.,\ } \Delta_{k+1}=\gamma^{-1}\Delta_k \},\\
B_k&=\mathbbm{1}\{ \Delta_k >\hat C_1 \|\nabla \phi(x_k)\|\}. 
\end{align*}
We will say that iteration $k$ is "true" if $I_k=1$. We make the following key assumption
\begin{assumption}\label{ass:prob_true_iter} There exists a $\theta\in (\frac{1}{2}, 1]$ such that
\[
{\mathbb P}\{I_k=1| {\mathcal F}_{k-1}\}\geq \theta. 
\]
\end{assumption}
This assumption is clearly made attainable by Lemma \ref{lem:haar_bound} with  $\theta\geq 
\frac{243}{443}>1/2$ and $\kappa_g=\sqrt{1-\frac{q}{10n}}$. Other ways of generating subspaces and ensuring Assumption \ref{ass:prob_true_iter} are of interest for future research.

Note that $\sigma(B_k)\subset {\cal F}_{k-1}$ and $\sigma(A_k)\subset {\cal F}_{k}$, that is the random variable 
$B_k$ is fully determined by the first $k-1$ steps of the algorithm, while $A_k$ is fully determined by the first $k$ steps. 

From Assumption \ref{ass:prob_true_iter} and Lemma \ref{eq:tr_sub_success_1} we have the following  dependency 
\[
A_k\geq I_k(1-B_k), 
\]
in other words, if iteration $k$ is true and trust region radius is sufficiently small, then the iteration is successful.

For the stochastic process generated by Algorithm~\ref{alg:tr_sub}  with random $Q_k$ that satisfy Assumption \ref{ass:prob_true_iter} the following dynamics hold. 

\begin{equation}\label{eq:proc1_Zk}
\Delta_{k+1}\geq \left \{ \begin{array}{ll} \gamma^{-1} \Delta_k& 
{\rm if\  } I_k=1\ {\rm and \ } B_k=0,
\\ \gamma \Delta_k & {\rm if\  }I_k=0\ {\rm and \ } B_k=0, \\
\gamma^{-1} \Delta_k & {\rm if\  }  A_k=1\ {\rm and \ } B_k=1, \\
 \gamma \Delta_k &    {\rm if\  }  A_k=0\ {\rm and \ } B_k=1,
\end{array}\right . \quad 
\phi(x_{k+1})\leq \left \{ \begin{array}{ll} \phi(x_k)-C_2\Delta_k^2 & {\rm if\  } I_k=1\ {\rm and \ } B_k=0,
 \\ \phi(x_k) & {\rm if\  }I_k=0\ {\rm and \ } B_k=0,\\
\phi(x_k)-C_2\Delta_k^2
&   A_k=1 \ {\rm and \ } B_k=1, \\ \phi(x_k) &   A_k =0 \ {\rm and \ } B_k=1. \end{array}\right . 
\end{equation}
Here $C_2$ is as in Lemma \ref{eq:tr_progress}. 

To bound the total number of iterations we first bound the number of iterations that are successful and with $\Delta_k\geq \gamma \hat C_1\epsilon$. For that let 
$\bar B_k=\mathbbm{1}\{ \Delta_k \geq \gamma \hat C_1 \epsilon\}$. Then from the dynamics \eqref{eq:proc1_Zk} we have the bound similar to \cite{cartis2018global}. 

\begin{lemma}\label{lem:bound_on_big2}
For any $l\in \{0,\ldots,K_{\epsilon}-1\}$ and for all realizations of Algorithm \ref{alg:tr_sub}, we have 
\[
\sum_{k=0}^l  \bar B_kI_kA_k \leq \sum_{k=0}^l  \bar B_k A_k \leq \frac{\phi(x_k)-\phi^\star}{C_2(\gamma \hat C_1\epsilon)^2},
\]
\end{lemma}
Another useful lemma that easily follows from the dynamics is as follows. 
\begin{lemma}\label{lem:bound_on_big}
For any $l\in \{0,\ldots,K_{\epsilon}-1\}$ and for all realizations of Algorithm \ref{alg:tr_sub}, we have 
\[
\sum_{k=0}^l B_k(1- A_k) \leq \sum_{k=0}^l\bar B_kA_k + \log_{\gamma}\left (\frac{\hat C_1 \epsilon}{\Delta_0}\right ). 
\]
\end{lemma}

The following result is shown in \cite{cartis2018global}  under Assumption \ref{ass:prob_true_iter}, 
\[
\mathbb{E}\left (\sum_{k=1}^{K_\epsilon-1} \bar B_k(I_k-1)\right ) \leq \frac {1-\theta}{\theta } \mathbb{E}\left(\sum_{k=1}^{K_\epsilon-1} \bar B_kI_k\right ), 
\]
from which the following lemma is derived. 
\begin{lemma}\label{lem:hittime2}
Under the condition that $\theta >1/2$, we have
\[
\mathbb{E}\left (\sum_{k=0}^{K_\epsilon-1}  B_k \right )\leq \frac{1}{2\theta-1}\left (\frac{\phi(x_k)-\phi^\star}{C_2(\gamma \hat C_1\epsilon)^2} + \log_{\gamma}\left (\frac{\hat C_1 \epsilon}{\Delta_0}\right )\right ) 
\]
\end{lemma}
Finally the following lemma is shown in \cite{cartis2018global} for the process obeying \eqref{eq:proc1_Zk}.
 \begin{lemma}\label{lem:hittime1}
 \[
  \mathbb{E}\left (\sum_{k=0}^{K_\epsilon-1}  (1-B_k) \right)\leq \frac{1}{2\theta }\mathbb(K_\epsilon).
  \]
  \end{lemma}

Putting these last two lemmas together we obtain the final expected complexity result. 
\begin{theorem}\label{thm:tr_sub_complexity}
    Let Assumption~\ref{assum:lip_cont}, Assumption \ref{assum:tr}  and Assumption \ref{ass:prob_true_iter} hold. Assume that  for all $k=0, 1, \ldots K_{\epsilon}-1$ $m_k(x_k+s)$ is a $\kappa_{ef}, \kappa_{eg}$-fully linear model of $\phi(x_k+s)$ on $B_{Q_k}(x_k,\Delta_k)$.  Then for some $\epsilon>0$ assuming that the initial trust-region radius $\Delta_0 \ge \gamma \hat C_1 \epsilon$
let   $ K_\epsilon $ be the random stopping for the event $\{\|\nabla \phi(x_k)\|\leq \epsilon\}$. We have the bound 
 \[
       {\mathbb E}\left [ K_\epsilon\right ]\leq \frac{2\theta}{(2\theta-1)^2} \left (\frac{2(\phi(x_0) - \phi^\star)}{C_2 (\gamma \hat C_1 \epsilon)^2} + \log_\gamma \frac{\hat C_1\epsilon}{\Delta_0}\right ) 
 \]
 where  $\hat C_1$ as in Lemma \ref{eq:tr_sub_success_1} and $C_2$ as in Lemma \ref{eq:tr_progress}.  
\end{theorem}

When  $m_k$ is constructed using \eqref{eq:subspace_ffd}, then 
the number of total zeroth order oracle calls is $(q+1)K_\epsilon$ , $\kappa_{eg}$ and $\kappa_{ef}$ depend on the subspace dimension $q$:
  $\kappa_{eg}=2\sqrt{q} L_Q$ and $\kappa_{ef}=\kappa_{eg}+\frac{L_Q+\kappa_{bhm}}{2}$. When $Q_k$ is drawn from Haar distribution with $q\geq 3$,  $\theta=\frac{243}{443}$ and $\kappa_g^2={\cal O}({\frac{n-q}{n}})$ 
 is dependent on  $n$. 
 Thus $\sqrt{1-\kappa_g^2} =\sqrt{\frac{q}{10n}}$. The resulting expected iteration complexity 
is $ {\cal O} \left(\frac{n}{\epsilon^2}\right )$ while each iteration requires only $q+1$ function evaluations, which gives the total worst case expected oracle complexity of 
\[
{\cal O}  \left(\frac{nq}{\epsilon^2}\right ).
\]

{\bf Conclusion:}  Algorithm \ref{alg:tr_sub} utilized finite difference gradient approximation with radius $\delta=\Delta_k$ and achieves 
${\cal O}  \left(\frac{nq}{\epsilon^2}\right )$ oracle complexity. In contrast  Algorithm \ref{alg:tr} has ${\cal O}  \left(\frac{n^2}{\epsilon^2}\right )$ worst case complexity
if it chooses $\delta=\Delta_k$ for the finite difference scheme. However, we note that  the lower  bound on $\Delta_k$ of Algorithm \ref{alg:tr} holds for all $k$ 
and is $\gamma C_1\epsilon$. On the other hand, the bound on $\Delta_k$ for Algorithm \ref{alg:tr_sub} is $\sqrt{\frac{q}{10n}}\gamma C_1\epsilon$ 
and it does not hold on all iterations (as seen in the analysis it holds "often enough", but not always). This complicates the analysis of
Algorithm \ref{alg:tr_sub} in the case of noisy zeroth order oracles, which we leave for future research. 

In summary, if using finite differences the complexity of Algorithm \ref{alg:tr_sub} is the same as that of Algorithm \ref{alg:tr} if the latter utilizes 
$\delta=\sqrt{\frac{q}{n}}\Delta_k$ in the finite difference scheme.  Our real interest is using subspace based trust region method is  the geometry-correcting method 
whose full-space version Algorithm \ref{alg:powell} has complexity ${\cal O}  \left(\frac{n^2}{\epsilon^2}\right )$. We address this next.

\subsection{Geometry-correcting algorithm in subspaces.} 
We are now ready to provide a geometry-correcting algorithm with the total worst case complexity matching the best known complexity.

\begin{algorithm}[ht]
\caption{~\textbf{Geometry-correcting algorithm in subspaces}}
\label{alg:powell_sub}
{\bf Inputs:} A zeroth-order oracle $f(x)= {\phi}(x)$, subspace dimension $q$,  a space of polynomials ${\cal P}$ of dimension $p$, $\Delta_0$,  $x_0$,  $\gamma \in (0,1)$ $\eta_1 >0$, $\eta_2 >0$, $\Lambda>1$.\\ 
{\bf Initialization} An initial orthonormal $Q_0\in {\mathbb R}^{n\times q}$, set ${\cal Y}_0\subset {\mathbb R}^q$ such that $|{\cal Y}_0| \leq p$, and the function values $f(x_0)$, $f(x_0 + y_i)$, $y_i \in {\cal Y}_0$. \\
\For{$k=0,1,2,\dots$} {
	\nl Build a quadratic model $m_k(x_k + s) $ as in \eqref{eq:model_def} using $f(x_k)$ and 
	$f(x_k + y_i)$, $y_i\in {\cal Y}_k$. 
	 \\
	\nl Compute a trial step $s_k=Q_kv_k$ and  ratio $\rho_k$  as in Algorithm \ref{alg:tr_sub}. \\
       \nl  {\em Successful iteration:} $\rho_k\geq \eta_1$ and     $\|g_k\|\geq \eta_2\Delta_k$.  Set $x_{k+1} = x_k+s_k$, $\Delta_{k+1}=\gamma^{-1}\Delta_k$.\\
 Replace the furthest interpolation point and shift ${\cal Y}_{k+1} = ({\cal Y}_k\setminus \{y_{j_k^*}\} \cup \{0\}) - v_k$ where 
  \[
{j_k^* =\arg \max_{j=1,\ldots, p} \|y_j\|.}
\]
Generate $Q_{k+1}\in {\mathbb R}^{n\times q}$ from the Haar distribution.\\
\nl {\em Unsuccessful iteration:} $\rho_k < \eta_1$ or  $\|g_k\|<\eta_2\Delta_k$. Set $x_{k+1}=x_k$ and perform the first applicable step:
\begin{itemize}
\item {\em Geometry correction by  adding a point:} If $|{\cal Y}_k|<p$,  ${\cal Y}_{k+1} ={\cal Y}_k \cup \{v_k\}$.
\item {\em Geometry correction by replacing a far point:} Let $j_k^* =\arg \max_{j=1,\ldots, p} \|y_j\|.$\\
 If  $\|y_{j_k^*}\|>\Delta_k$ $\Rightarrow$   {${\cal Y}_{k+1} = {\cal Y}_k\setminus \{y_{j_k^*}\} \cup \{v_k\}$. }
\item {\em Geometry correction by replacing a "bad" point:} Otherwise, construct the set Lagrange Polynomials  
$\{\ell_j(x), i=1, \ldots, p\}$ in  ${\cal P}$ for the set ${\cal Y}_k$ and  find max  value 
\[
{(i_k^*,v_k^*) =\arg \max_{j=1,\ldots, p, v\in B(0,\Delta_k)} |\ell_j(v)|.}
\]
 If $|\ell_{i_k^*}(v_k^*)|>\Lambda$, compute  $f(x_k + Q_kv_k^*)$, ${\cal Y}_{k+1} = {\cal Y}_k\setminus \{y_{i_k^*}\} \cup \{v_k^*\}$. 
\item {\em Geometry is good:} Otherwise $\Delta_{k+1} =\gamma \Delta_k$. Generate $Q_{k+1}\in {\mathbb R}^{n\times q}$ from the Haar distribution.
\end{itemize}
%
}
\end{algorithm}

The key observation is that not only we do not decrease $\Delta_k$ but we also do not generate new subspace matrix $Q_k$ on geometry-correcting iterations. Thus there is no randomness involved in these iterations and the analysis of Algorithm \ref{alg:tr_sub} carries over essentially without change if we restrict all statements to iterations 
$k$ that are {\em not } geometry-correcting, that is $k\in {\cal S}_\epsilon\cup{\cal U}_\epsilon^d$.  Thus the bound on   ${\mathbb E}\left [ K_\epsilon\right ]$ in Theorem \ref{thm:tr_sub_complexity}
applies as the bound on ${\mathbb E}|{\cal S}_\epsilon\cup{\cal U}_\epsilon^d|$, which equals to the  bound on the total number of zeroth order oracle calls for all $k\in {\cal S}_\epsilon\cup{\cal U}_\epsilon^d$ (since on such iteration the zeroth order oracle is called only once). 

We thus can utilize the results of Theorem \ref{thm:geom_correct_total_lin} which bounds the number zeroth order oracle calls during consecutive geometry-correcting  steps as $3q$ if linear interpolation is used. For a higher degree interpolation we can apply Theorem \ref{thm:plogp} to bound the total number of zeroth order calls during 
consecutive  geometry-correcting iterations as ${\cal O}(p\log p)$ where $p$ is the dimension of the space of polynomials defined on ${\mathbb R}^q$, thus $p$ is polynomial in $q$. 
Finally, to bound $\kappa_{eg}$ and $\kappa_{ef}$ in $\hat C_1$, we can employ Theorem \ref{thm:Lambda_to_kappaeg} for linear interpolation or more general result similar to
Theorem \ref{thm:lambda_to_kappaeg_book} from \cite{DFOBook} for higher order interpolation to obtain  $\kappa_{eg}={\cal O}(\sqrt{q})$ if $\Lambda \approx 1$ 
or, more generally, $\kappa_{eg}$ has polynomial dependency on $q$.  The the total complexity is ${\cal O}  \left(\frac{nq^\alpha}{\epsilon^2}\right )$ for some $\alpha>0$.


\section{Conclusions}
\label{sec:conclusion}
We have shown that a practically efficient model based trust region DFO method such as  Algorithm \ref{alg:powell} and its subspace version Algorithm \ref{alg:powell_sub} have oracle complexity that is  comparable with  other known (and less practical) derivative free methods. The are many further practical  improvements that can be applied to Algorithm \ref{alg:powell} and Algorithm \ref{alg:powell_sub} that do not affect complexity but complicate the exposition. These can include avoiding computing $\rho_k$ 
 when $\|g_k\|\leq  \eta _2\Delta_k$, adding not random directions to subspaces and many others.

\subsection*{Acknowledgements}
\label{sec:Acknowledgements}
This work was partially supported by  ONR award N00014-22-1-215 and the Gary C. Butler Family Foundation.

\bibliographystyle{plain}
\bibliography{references}

\end{document}